\documentclass[12pt]{amsart}
\usepackage{amsmath, amssymb}
\newfont {\cyr} {wncyr10}
\renewcommand{\labelenumi}{{(\roman{enumi})}}
\usepackage{color}

\usepackage{amsfonts}
\usepackage{amsmath}
\usepackage{amssymb}
\usepackage{setspace}
\usepackage{a4}
\usepackage{multicol}

\newtheorem{theorem}{Theorem}[section]
\newtheorem{lemma}[theorem]{Lemma}
\newtheorem{proposition}[theorem]{Proposition}

\newtheorem{definition}[theorem]{Definition}

\newcounter{claim}[theorem]

\newcounter{cclaim}[theorem]

\def \udot {{}^{\textstyle .}}

\newcommand{\E}{\mathrm{E}}\newcommand{\SU}{\mathrm{SU}}
\newcommand{\F}{\mathrm{F}}\newcommand{\M}{\mathcal{M}}
\newcommand{\G}{\mathrm{G}}

\newcommand{\C}{\mathcal{C}}
\newcommand{\Q}{\mathrm{Q}}

\newcommand{\Aut}{\mathrm{Aut}}

\newcommand{\Out}{\mathrm{Out}}

\newcommand{\Syl}{\mathrm{Syl}}\newcommand{\syl}{\mathrm{Syl}}

\newcommand{\GF}{\mathrm{GF}}
\newcommand{\GL}{\mathrm{GL}}
\newcommand{\Sp}{\mathrm{Sp}}
\newcommand{\SL}{\mathrm{SL}}

\newcommand{\PSL}{\mathrm{PSL}}\newcommand{\PSp}{\mathrm{PSp}}
\newcommand{\Sym}{\mathrm{Sym}}
\newcommand{\Alt}{\mathrm{Alt}}

\def \qedc {$\hfill \blacksquare$\newline}

\def \L {\hbox {\rm L}}

\def \syl {\hbox {\rm Syl}}\def \Syl {\hbox {\rm Syl}}

\def \ov {\overline}

\def \wt {\widetilde}

\def \Aut{ \mathrm {Aut}}

\def \Out{\mbox {\rm Out}}

\def \Fi {\mbox {\rm Fi}}

\def \J{\mbox {\rm J}}

\def \B{\mbox {\rm B}}

\def \M{\mbox {\rm M}}

\def \Co {\mbox {\rm Co}}

\def \PSU {\mbox {\rm PSU}}
\def \GSp {\mbox {\rm GSp}}

\begin{document}
\renewcommand{\labelenumi}{(\roman{enumi})}

\title  {An identification theorem for the sporadic simple groups $\F_2$ and $\M(23)$}
 \author{Chris Parker}
  \author{Gernot Stroth}

\address{Chris Parker\\
School of Mathematics\\
University of Birmingham\\
Edgbaston\\
Birmingham B15 2TT\\
United Kingdom} \email{c.w.parker@bham.ac.uk}

\address{Gernot Stroth\\
Institut f\"ur Mathematik\\ Universit\"at Halle - Wittenberg\\
Theordor Lieser Str. 5\\ 06099 Halle\\ Germany}
\email{gernot.stroth@mathematik.uni-halle.de}

\email {}

\date{\today}

\begin{abstract} We identify the sporadic groups $\M(23)$ and $\F_{2}$ from the approximate  structure of the centralizer of an element of order $3$.

 \end{abstract}

\maketitle \pagestyle{myheadings}

\markright{{\sc sporadic groups  }} \markleft{{\sc Chris Parker and Gernot Stroth}}

\section{Introduction}

The sporadic simple group known as the Baby Monster and denoted here by $\F_2$ is defined to be the finite group $G$ having  involutions $\pi$ and $\sigma$ such that $C_G(\pi) \cong  (2\udot {}^2\E_6(2)){:}2$ and $C_G(\sigma)$ has type $2^{1+22}_+.\Co_2$. Here we say that $H$ has type $2^{1+22}_+.\Co_2$ if $F^*(H)$ is extraspecial of order $2^{23}$ and $+$-type and $H/F^*(H) \cong \Co_2$. In \cite{Segev} it is shown that there is a unique group having  such involution centralizers.

For  a definition of the Fischer group  $\M(23)$, we follow Aschbacher and  define $\M(23)$ to be the  finite
group $G$ with an involution $d$ which is not weakly closed in a Sylow $2$-subgroup of $G$ and such that
$C_G(d)/\langle d \rangle \cong \M(22)$. In \cite[Theorem 32.1]{Aschbacher3T}, Aschbacher shows that there is a
unique such group  and that it  is a $3$-transposition group.

This article characterizes  $\M(23)$ and $\F_2$ from the structure of a centralizer of an element of order $3$
and certain fusion data. This work is one of a series of papers where this type of problem  is addressed.
Particularly relevant to this article are \cite{ParkerRowley,PSS,PS1,F4}. We refer the reader to \cite{PS1} where we
motivate this study by describing certain configurations which appear when classifying the finite simple groups.
Here we content ourselves with the fact that the large sporadic simple groups which are the focus of this paper
are interesting in their own right. We note however, that the main theorem of this article has immediate application in \cite{PS2}.

\begin{definition} \label{dd1} A group   $X$ is similar to a $3$-centralizer in a group of type  $\M(23)$ provided
\begin{enumerate}
\item $Q=F^*(X)= O_3(X)$ is extraspecial of type $3^{1+8}_+$ and $Z(Q) =Z(X)$;
\item  $F^*(X/Q)= O_2(X/Q)$ is extraspecial of type $2^{1+6}_-$; and
\item $X/O_{3,2}(X)$ is isomorphic to the centralizer of a $3$-central element in $\PSp_4(3)\cong \Omega_6^-(2)$.
\end{enumerate}
\end{definition}

So in Definition~\ref{dd1}  we are saying that $X$ has shape $3^{1+8}_+. 2^{1+6}_-.3^{1+2}_+.\Q_8.3$.

\begin{definition} \label{dd2} A group   $X$ is similar to a $3$-centralizer in a group of type  $\F_2$ provided
\begin{enumerate}
\item $Q=F^*(X)= O_3(X)$ is extraspecial of type $3^{1+8}_+$ and $Z(Q) =Z(X)$;
\item  $F^*(X/Q)= O_2(X/Q)$ is extraspecial of type $2^{1+6}_-$; and
\item $X/O_{3,2}(X) \cong \Omega_6^-(2)$.
\end{enumerate}
\end{definition}

In this definition we have that $X$ has shape $3^{1+8}_+.2^{1+6}_-.\Omega_6^-(2)$. That $\M(23)$ and $\F_2$ have centralizers of the described shapes can be seen in \cite[Tables 5.3u and 5.3.y]{GLS3}. In both groups  the elements of order $3$ being centralized are $3$-central and inverted in their normalizers.  Recall that for a group $G$ with subgroups $X \le Y \le G$, we say that $X$ is \emph{weakly closed} in $Y$ with respect to $G$ if $X$ is the unique $G$-conjugate of $X$ contained in $Y$.

Our main theorems are as follows.

\begin{theorem}\label{Main1}
Suppose that $G$ is a finite group, $H \le G$ is similar to a $3$-centralizer in a group of type $\M(23)$, $Z =
Z(F^*(H))$ and $H= C_G(Z)$. If $Z$ is not weakly closed  in a Sylow 3-subgroup  $S$ of $H$ with respect to $G$, then
$G$ is isomorphic to $\M(23)$.
\end{theorem}

\begin{theorem}\label{Main2}
Suppose that $G$ is a finite group, $H \le G$ is similar to a $3$-centralizer in a group of type $\F_2$, $Z = Z(F^*(H))$
and $H= C_G(Z)$. If $Z$ is not weakly closed    in a Sylow 3-subgroup  $S$ of $H$ with respect to $G$, then $G$ is isomorphic to $
\F_2$.
\end{theorem}

We remark that all the lemmas and propositions proved in the proof of our main theorems are results about groups with the given $3$-centralizers and so can be used as a source of facts about the groups $\F_2$ and $\M(23)$.

In Section~2,  we present some background lemmas as well as introduce the definitions of a  group $X$ being similar to $3$-centralizers in the groups $\Co_2$,  ${}^2\E_6(2)$ and $\M(22)$. In Section 3, we construct groups $M$ and $U$,  similar to  $3$-centralizers in $\M(23)$ and $\F_2$ respectively, as semidirect products of an extraspecial group of order $3^9$ and a carefully constructed  subgroup of $\Sp_8(3)$.  In particular, in this section we describe  a certain involution  $\pi$.

It turns out that any group which is similar to a  $3$-centralizer of type $\F_2$ is isomorphic to $U$; however,  the isomorphism type of groups which are similar to $3$-centralizers in $\M(23)$ is not uniquely determined.  In the rather technical Section 4, we determine various properties of groups $H$ which are similar to $3$-centralizers of type $\M(23)$ and $\F_2$ in a context where they are  embedded in some finite group $G$. In particular, we determine $C_H(\pi)$ in Lemma~\ref{c3} where we show that $C_H(\pi)/\langle \pi\rangle$ is similar to a $3$-centralizer in $\M(22)$ if $H$ is similar to a $3$-centralizer in $\M(23)$ and $C_H(\pi)/\langle \pi\rangle$ is similar to a $3$-centralizer in ${}^2\E_6(2)$ if $H$ is similar to a $3$-centralizer in $\F_2$. In addition, in this section, when
$H$ is similar to a $3$-centralizer in $\F_2$, we   locate a further involution $r\in H$ and, in Lemma~\ref{CR}, show that $C_H(r)/O_2(C_H(r))$ is similar to a $3$-centralizer in $\Co_2$.

Suppose now that $G$ is a group which satisfies  the hypotheses of Theorems~\ref{Main1} or \ref{Main2} and let $S$, $H$ and $Z$ be as in the  statements of these theorems.  The objective of Section~5 is to prove that $C_G(\pi)/\langle \pi\rangle\cong \M(22)$ when $H$ is similar to a $3$-centralizer in $\M(23)$ and  $C_G(\pi) \cong (2\udot{}^2\E_6(2)){:}2$ when $H$ is similar to a $3$-centralizer in $\F_2$. We intend to apply the appropriate identification theorems from \cite{PSS}. Because of the conclusions from Section~4, to apply \cite{PSS} we need to show that $Z$ is not weakly closed in $C_S(\pi)$ with respect to $C_G(\pi)$. This is the main objective of Section~5. The proof of this fact relies heavily on the detailed descriptions of groups similar to $3$-centralizers of type $\M(23)$ and $\F_2$ obtained in Sections 3 and 4 and is probably the highlight of this article. By the end of Section 5, we know,   if $G$ satisfies the hypothesis of Theorem~\ref{Main1}, then $C_G(\pi)/\langle \pi \rangle \cong \M(22)$. At this stage of our arguments, it is easy to see that $\pi$ is not weakly closed in a Sylow $2$-subgroup of $G$. Therefore, in this case,   $G$ is isomorphic to $\M(23)$ and the proof of Theorem~\ref{Main1} is complete.

To prove Theorem~\ref{Main2} we must determine $C_G(r)$. This is the goal of the remaining three sections of the paper. The first of these, Section~6, is very short and shows that there is an element $\rho$ of $G$ with  $N_G(\langle \rho\rangle)\cong \Sym(3) \times \Aut(\M(22))$ (Lemma~\ref{Crho}).
The proof of this statement is a relatively straightforward application of \cite[Theorem 2]{PSS}.  In Section~7 the main objective is to prove Proposition~\ref{NE} which states that $C_G(r)$ contains an extraspecial subgroup $E$ of plus type and order $2^{23}$ and centre $\langle r \rangle $ such that $N_G(E)/E \cong \Co_2$. Our precise description of the involution $r$ means that we can determine its centralizer in $C_G(\pi)$. Thus we begin the determination of $C_G(r)$ knowing that $C_{C_{G}(\pi) }(r)/\langle r \rangle$ has shape $2^{1+20}_+.\PSU_6(2).2$. The determination  of $C_G(r)$ uses this fact,  exploits \cite{ParkerRowley} and the fact that  $C_H(r) /O_2(C_H(r))$ is similar to a $3$-centralizer in $\Co_2$ which was demonstrated in Section~4.  The final result of Section 7 asserts $N_G(E)$ is strongly $3$-embedded in $C_G(r)$. With this to hand in Section~8, we show that $N_K(E) = C_G(r)$ and thus we complete the proof of Theorem~\ref{Main2}.

Throughout this article we follow the now standard Atlas \cite{Atlas} notation for group extensions. Thus $X\udot
Y$ denotes a non-split extension of $X$ by $Y$, $X{:}Y$ is a split extension of $X$ by $Y$ and we reserve the
notation $X.Y$ to denote an extension of undesignated type. Our group
theoretic notation is mostly standard and follows that in  \cite{GLS2} for example. For odd primes $p$, the
extraspecial groups of exponent $p$ and order $p^{2n+1}$ are denoted by $p^{1+2n}_+$. The extraspecial $2$-groups
of order $2^{2n+1}$ are denoted by $2^{1+2n}_+$ if the maximal elementary abelian subgroups have order $2^{1+n}$
and otherwise we write $2^{1+2n}_-$. We expect our notation for specific groups is self-explanatory; however, we remark that $\M(22)$,  $\M(23)$ and are the sporadic simple groups discovered and constructed by Fischer  which are also denoted by $\Fi_{22}$, $\Fi_{23}$ and $\Fi_{24}'$ and $\F_2$  is the Baby Monster simple group. For a subset
$X$ of a group $G$, $X^G$ denotes that set of $G$-conjugates of $X$. If $x, y \in H \le  G$, we  write
$x\sim _Hy$ to indicate that $x$ and $y$ are conjugate in $H$.  Often we shall give suggestive descriptions of groups which indicate the isomorphism type of certain composition factors. We refer to such descriptions as the \emph{shape} of a group. Groups of the same shape have normal series with isomorphic sections. As an example, we refer back to the description of a $3$-centralizer of type  $\F_2$. We have already said that it has shape $3^{1+8}_+.2^{1+6}_-.\Omega_6^-(2)$ and this means that there is a normal subgroup isomorphic to $3^{1+8}_+$, a section isomorphic to $2^{1+6}_-$ and a quotient isomorphic to $\Omega_6^-(2)$ it says nothing about the action of these groups on the various sections.  We use the symbol $\approx$ to indicate the shape of a group.

\medskip

\noindent {\bf Acknowledgement.}  The first author is  grateful to the DFG for their support and thanks the mathematics department in Halle for their hospitality.

\section{Preliminary lemmas, facts and definitions}

\begin{lemma}\label{fusion}  Suppose that $p$ is a prime, $G$ is a group, $H \le G$ and $x$ is a $p$-element of $G$. Assume that the following statements  hold.
\begin{enumerate}
\item [(a)] $C_G(x) \le H$; and
\item[(b)] $x^G \cap H= x^H$.
\end{enumerate}
Then the following statements also hold.
\begin{enumerate}
\item  If $x \in K \le H$, then $N_G(K) \le H$; and
\item  For  $a \in G$, $a^G \cap C_G(x) = a^H \cap C_G(x)$.
\end{enumerate}
\end{lemma}

\begin{proof} Assume that $x \in K \le H$ and $y \in N_G(K)$. Then
$x^y \in K \le H$ and so, by (a), there exits $h \in H$ such that  $x^{yh}= x$. But then $yh \in C_G(x)\le H$ and
so $y \in H$. Therefore $N_G(K) \le H$ and (i) holds.

Suppose that $a \in C_G(x)$,  $g\in G$ and $b=a^g \in  a^G \cap  C_G(x)$. Let $T \in \Syl_p(C_H(a))$ and $T_1 \in
\syl_p(C_H(b))$ with $x \in T$ and $x\in T_1$. Then, by (i), $T \in \syl_p(C_G(a))$ and $T_1 \in \syl_p(C_G(b))$.
Since $T^g \in \syl_p(C_G(b))$, there exists $w\in C_G(b)$ such that $T^{gw} =T_1$. Now we have $x^{gw} \in
T_1\le H$ and so using (b) there exists $h \in H$ such that $x^{gwh} = x$. Therefore, by (a),  $gwh \in C_G(x)
\le H$ and so $gw \in H$. Since $a^{gw}= b^w=b$, we have  $b \in a^H\cap C_G(x)$. Thus (ii) holds.
\end{proof}

For a group $X$ with subgroups $A \le Y \le X$  we say that $A$ is \emph{strongly closed  in $Y$ with
respect to $X$} provided $A^x \cap Y \le A$ for all $x \in X$.

\begin{lemma}\label{Gold} Suppose that $K$ is a group, $O_{2'}(K)=1$,  $E$ is an abelian $2$-subgroup of $K$ and $E$ is strongly closed in $N_K(E)$. Assume that $F^*(N_K(E)/E)$ is a non-abelian simple group.
 Then  $K= N_K(E)$.
\end{lemma}

\begin{proof} Set  $ L= \langle E^{ K}\rangle$.   Since $O_{2'}(K)=1$, we have $O_{2'}(L)=1$.
By Goldschmidt \cite[Theorem A]{Goldschmidt}, we have $ L=  O_2( L)E( L)$ and $ E=  O_2( L)\Omega_1( T)$ where $ T \in \syl_2( L)$ contains $ E$. If $E( L)=1$, then $ E$ is normal in $ K$ and we are done.  Thus $E( L) \neq 1$.
 Goldschmidt additionally states that  $E( L)$ is a direct product of simple groups of type $\PSL_2(q)$, $q \equiv 3,5 \pmod 8$, ${}^2\G_2(3^a)$,  $\SL_2(2^a)$, $\PSU_3(2^a)$, ${}^2\B_2(2^a)$ for some natural number $a$, or the sporadic simple group $\J_1$.
It  follows from the structure of these groups that $N_{ L}( E)$ is a soluble group which is not a $2$-group.
On the other hand, $N_{ L}( E)=  L \cap {N_K(E)}$ is a normal subgroup of $ {N_K(E)}$.  Since $F^*(N_K(E)/E)$ is a non-abelian simple group, we now have $N_{ L}( E)$ is non-soluble which is  a contradiction. This proves the lemma.
\end{proof}

The next lemma is well known (see \cite{JP}), but we sketch a proof. Note that $\Sp_6(2)$ has a unique
$8$-dimensional irreducible module in characteristic $2$ (see \cite[5.4]{Aschbacher2E6} for example).

\begin{lemma}\label{cohom} Suppose
that $G \cong \Sp_6(2)$ and $V$ is the $8$-dimensional irreducible $\GF(2)G$-module. Then $\mathrm H^1(G,V)=0$.
\end{lemma}

\begin{proof} Let $d \in G$ be an element of order $3$
with centralizer $3 \times \Sp_4(2)$. Then $d$ acts fixed point freely on $V$. Suppose that $W > V$ and $W/V$ is
$1$-dimensional. Then, letting $S$ be a Sylow $3$-subgroup containing $d$, we have   $C_W(d) = C_W(S)$ is
$1$-dimensional and is invariant under $G=\langle S,C_G(d)\rangle$. Hence $\mathrm H^1(G,V)=0$.
\end{proof}
Our final identification of $\F_2$ requires us to have information about the action of $\Co_2$ on its natural
$22$-dimensional $\GF(2)$-representation. We remark that we do not use the fact that this module is unique.
Before we discuss this module however, we need some facts about $\PSU_6(2){:}2$ and its irreducible $20$-dimensional module over $\GF(2)$.
Let $X= \PSU_6(2){:}2$. Then, by  \cite[(23.2)]{Aschbacher3T} and  \cite[Proposition 4.9.2]{GLS3},  $X'$ has three $X$-classes of involutions with representatives and $t_1, t_2$ and $t_3$ and $X\setminus X'$ contains two $X$-classes of involutions with representatives  $t_4$ and $t_5$. Furthermore, we have
\begin{eqnarray*}
C_X(t_1) &\approx& 2^{1+8}_+{:}\SU_4(2). 2; \\
C_X(t_2) &\approx& 2^{4+8}.(\Sym(3) \times \Sym(3)).2; \\
C_X(t_3) &\approx& 2^{9}.3^2.\Q_8.2 \le 2^9{:}\L_3(4).2;\\
C_X(t_4) &\approx& 2\times \Sp_6(2); \text{ and}\\
C_X(t_5) &\approx& 2 \times (2^5{:}\Sp_4(2)).\\
\end{eqnarray*}

\begin{proposition}\label{Vaction} Suppose that $X = \PSU_6(2){:}2$ and $V$ is the irreducible $\GF(2)X$-module of dimension $20$.
  Then the following hold:
\begin{enumerate}
\item $\dim C_V(t_1) = 14$, $[V,t_1]$ is the orthogonal module and $C_V(t_1)/[V,t_1]$ is
the unitary module for $C_X(t_1)/O_2(C_X(t_1)) \cong \SU_4(2)$; \item  $\dim C_V(t_2) = 12$ and $C_V(t_2)/[V,t_2]$ is the natural  orthogonal module for  $C_X(t_2)/O_2(C_X(t_2)) \cong \Omega_4^+(2)$;
\item $\dim C_V(t_4) = 14$ and
 $[V,t_4]$ is the symplectic module and $C_V(t_4)/[V,t_4]$ is the spin module for  $C_X(t_4)/O_2(C_X(t_4)) \cong \Sp_6(2)$; and
\item  $\dim C_V(t_3) = \dim C_V(t_5)= 10$.
\end{enumerate}

\end{proposition}
\begin{proof}
See  \cite[Proposition 2.2]{PSS}.
\end{proof}

\begin{lemma}\label{Co2} Suppose that $G \cong \Co_2$, $H
\le G$ is isomorphic to $\PSU_6(2){:}2$ and $V$ an irreducible $\GF(2)G$-module of dimension $22$  in
which $H$ fixes a $1$-dimensional subspace and has a $20$-dimensional composition factor. Let $\widehat G$ be a
group with $\widehat G/O_2(\widehat G) \cong G$ and $O_2(\widehat G)$ isomorphic to $V$ as a $\widehat
G/O_2(\widehat G)$-module. Then the following hold:
\begin{enumerate}
\item $G$ has three conjugacy classes of involutions with representatives $s_1$, $s_2$, $s_3$ and centralizers
\begin{eqnarray*}C_{G}(s_1) &\approx & 2^{1+8}_+.\Sp_6(2);\\ C_{G}(s_2) &\approx & (2^{1+6}_+ \times
2^4).\Alt(8); \text{ and }\\C_{G}(s_3) &\approx & 2^{10}.\Aut(\Alt(6)).\end{eqnarray*}

\item Every involution of $G$ is conjugate to an involution of $H'$.
\item Every involution of $\widehat G\setminus O_2(\widehat G)$ centralizes an element of order $3$.
\item If $S$ is a Sylow 2-subgroup of $\widehat G$, then $O_2(\widehat G) = J(S)$.

\end{enumerate}
\end{lemma}

\begin{proof} The conjugacy classes of involutions in $G$ and their centralizers   are given in \cite[Table 5.3k]{GLS3}.  Thus (i) holds.

 We adopt the notation $t_1$, $t_2$, $t_3$, $t_4$ and $t_5$ for involutions in $H\cong \PSU_6(2){:}2$
introduced before Proposition~\ref{Vaction}  and denote the $20$-dimensional composition factor  of $V$ by $W$.
Obviously $s_1 \sim _Gt_1$. From  Proposition~\ref{Vaction} (i), we have $\dim C_W(t_1)/[W,t_1]= 8$, $\dim
C_W(t_2)/[W,t_2] = 4$ and $\dim C_W(t_3)/[W,t_3]=0$ and in each instance $C_H(t_i)$ acts irreducibly on
$C_W(t_i)/[W,t_i]$. It follows that the corresponding  quotients in $V$ can only change dimension by increasing at
most  $2$. Therefore $t_1$, $t_2$ and $t_3$ do not fuse in $G$. This shows  (ii) holds and, moreover, $s_1 \sim _Gt_1$. As
there is no subgroup isomorphic to $3^2{:} \Q_8$ in $\Alt(8)$ we see with Proposition~\ref{Vaction}  that
$s_2 \sim_G t_2$ and therefore $s_3\sim_Gt_3$.

For (iii) we identify $\widehat G/O_2(\widehat G)$ with $G$ and $O_2(\widehat G)$ with $V$. We are required to determine the involutions in each coset $Vs_i$,
$i=1,2,3$, and show that each one is centralized by a $3$-element.

Since $C_H(t_1)$ acts as $\SU_4(2)$ on $C_W(t_1)/[W,t_1]$, we infer that $|C_V(s_1)/[V,s_1]|=2^8$ or $2^{10}$ and
admits a faithful action of $C_G(s_1)/O_2(C_G(s_1))\cong \Sp_6(2)$. By Lemma~\ref{cohom}  the $8$-dimensional module for
$\Sp_6(2)$ splits with the trivial module so   the  orbits of $C_G(s_1)$ on
$(C_{V}(s_1)/[V,s_1])s_1$ have lengths $1, 135,$ and $120$. In particular, every involution in $Vs_1$ is
centralized by an element of order $3$.

Using \cite[Theorem 1]{MS}, we obtain that $C_V(s_2)$ has dimension $14$ and  $C_{V}(s_2)/[V, s_2]$  is the
irreducible $6$-dimensional orthogonal module for $\Alt(8)$. Therefore on the coset $(C_{V}(s_2)/[V,s_2])s_2$ we have
$\Alt(8)$-orbits of length 1, 35 and 28 or 8 and 56 depending upon whether the $\Alt(8)$-space $(\langle
C_{V}(s_2),s_2\rangle/[V,s_2])$ splits as a direct sum or is indecomposable as an $\Alt(8)$-module. In both cases
any involution in $Vs_2$ is centralized by a Sylow 3-subgroup of $C_{G}(s_2)$.

Finally consider the coset $Vs_3$. Then,   as seen above, we have $[W,s_3] = C_W(s_3)$. Hence $|C_{V}(s_3)/[V,
s_3]| \leq 4$. In particular, $[O^2(C_{G}(s_3)), C_{V}(s_3)] \leq [V,s_3]$ and so again all involutions in this
coset are centralized by a Sylow 3-subgroup of $C_{G}(s_3)$.

For part (iv) we simply note that the largest quadratically acting subgroup of $G$ has order $4$ by \cite[Lemma
2.19]{MS} and thus the result follows from our preceding remarks on centralizers of involutions on $V$ in the
proof of (iii).
\end{proof}

\begin{lemma}\label{perp} Suppose that $p$ is a prime, $E$ is an extraspecial $p$-group and $y\in \Aut(E)$ centralizes $Z(E)$.
Let $F$ be the preimage of $C_{E/Z(E)}(y)$. Then $C_E(F)=[E,y]$ and $C_E([E,y]) = F$.
\end{lemma}

\begin{proof} Recall from \cite[III(13.7)]{Huppert} that $E/Z(E)$ supports a non-degenerate symplectic form which is defined by commutation.
Thus for $Z(E) \le W \le E$, $(W/Z(E))^\perp = C_E(W)/Z(E)$.

We have $[F,y,E]\le [Z(E),E]=1$ and $[E,F,y]\le [Z(E),y]=1$. Thus $[E,y,F]=1$ by the Three Subgroup Lemma and so
$[E,y]/Z(E)\le (F/Z(E))^\perp$. Since $|[E,y]/Z(E)|= |E:F|$, we now have $[E,y]/Z(E)= (F/Z(E))^\perp$. Hence $C_E(F)=
[E,y]$ and $C_E([E,y])= F$.
\end{proof}

For the convenience of the reader we recall the following definitions from \cite{ParkerRowley} and \cite{PSS} where it is   proved that the  groups with socle $\Co_2$, ${}^2\E_6(2)$ and $\M(22)$
are uniquely determined by their $3$-centralizers once a weak closure condition is imposed.

\begin{definition}\label{co2} We say that $X$ is similar to a $3$-centralizer in $\Co_2$ provided
$O_3(X)$ is extraspecial of order $3^5$, $O_2(X/O_3(X))$ is extraspecial of order $2^5$ and $X/O_{3,2}(X) \cong
\Alt(5)$.
\end{definition}

\begin{definition}\label{d1} Suppose that $X$ is a group with $Q=F^*(X)$  extraspecial of order $3^{1+6}$ and $Z(F^*(X))
=Z(X)$. Then\begin{enumerate}\item  $X$ is similar to a $3$-centralizer in a group of type ${}^2\E_6(2)$ provided
$O_{2,3}(X)/Q \cong \Q_8\times \Q_8\times \Q_8$; and
\item $X$ is similar to a $3$-centralizer in a group of type  $\M(22)$
provided $O_{2,3}(X)/Q$ acts on $F^\ast(X)$ as a subgroup of order $2^7$ in $\Q_8\times \Q_8\times \Q_8$ which
contains $Z(\Q_8\times \Q_8\times \Q_8)$.
\end{enumerate}
\end{definition}

Finally for this section we collect the following facts about ${}^2\E_6(2)$ from \cite{PSS}.

\begin{proposition}\label{2e62facts} Suppose that $X \cong {}^2\E_6(2){:}2$ and that $J =C_X(y)$ is the centralizer of a $3$-central element $y$ in $X$.  Let $I = O_3(J)$ and $S \in \Syl_2(X)$. Then
\begin{enumerate}
\item $J$ is similar to a $3$-centralizer in ${}^2\E_6(2)$;
\item $y$ is inverted in $X$;
\item $y$ is $X$-conjugate to an element of $I \setminus \langle y\rangle$;
\item There is an element $\tau\in I$ such that $C_X(\tau)\cong 3 \times \PSU_6(2).2$  and $C_S(\tau) \in \Syl_3(C_X(\tau))$;
\item  The involutions $s \in O_{3,2}(J)$ such that $C_I(s) \cong 3^{1+4}_+$ are $2$-central in $X$; and
\item If $s$ is a $2$-central involution in $X$, then $C_X(s) \approx 2^{1+20}_+.\PSU_6(2).2$, $F^*(C_X(s))$ is extraspecial of order $2^{21}$ and $C_X(s)$ acts absolutely irreducibly on $F^*(C_X(s))/\langle s \rangle$.
\end{enumerate}
\end{proposition}

\begin{proof} As remarked after the main theorems in \cite{PSS}, as a consequence of
\cite[Theorem 1]{PSS}, all the lemmas of that paper provide statements about the groups with socle ${}^2\E_6(2)$.
In  particular, (i) holds, part (ii) comes from \cite[Lemma 4.6(v)]{PSS}, (iii) from  \cite[Lemma 4.5]{PSS}, (iv)
from \cite[Lemma 7.1 (i) and (iii)]{PSS} and (v) and (vi) from \cite[Theorem 10.4]{PSS}.
\end{proof}

\section{A construction of $3$-centralizers of type $\F_2$ and $\M(23)$}

Let $W$ be  the natural $2$-dimensional symplectic $\GF(3)$-module for $\GL(W)\cong \GSp(W)$ with natural
symplectic basis $\{e,f\}$ and associated symplectic form $(\;,\;)$. The space  $V= W \otimes W \otimes W$
supports a symplectic form defined by $$(v_1\otimes v_2\otimes v_3,w_1\otimes w_2 \otimes w_3) =
\prod_{i=1}^3 (v_i,w_i).$$  Furthermore $V$  admits $\bar X^*=\GSp(W) \wr \Sym(3)$ where a base group element
$(x,y,z)$ acts on the vector$(v_1\otimes v_2\otimes v_3)$  by $$(v_1\otimes v_2\otimes v_3)(x,y,z) = (v_1x\otimes
v_2y\otimes v_3z)$$ and the permutation group permutes the tensor factors. We let $\bar X$ be the subgroup of
$\bar X^*$ which preserves the form on $V$. Thus $\bar X$ has index $2$ in $\bar X^*$ and has shape $(\Q_8
\times \Q_8\times \Q_8).3^3.2^2.\Sym(3)$.

With this action we obtain a representation of $\bar X$ into $\Sp(V)$ with  kernel $K=\langle (-I,-I,I),
(I,-I,-I)\rangle$  where $I$ is the identity transformation in $\Sp(W)$. Thus $\bar X$ has image $X \cong \bar
X/K$ in $\Sp(V)$. We have $R=O_2(X) \cong 2^{1+6}_-$ and $X/R  \approx 3^3.\Sym(4)$.

Denote by $d$  the element of $\GSp(W)$ which centralizes $e$ and sends $f$ to $e+f$. Then  $d$ has order $3$.
Let   $d_1 = (d,I,I)$, $d_2 = (d,d^{-1},I)$,  $d_3 = (d,d,d)$, $\sigma= (-I,-I,-I)$, $\pi$ be the element which
permutes the base group of $\bar X$ as $(1,2)$ and $\tau$ the element which permutes the base group as $(1,2,3)$.
We let $D = \langle d_1,d_2,d_3\rangle$ and note that $D\langle \tau\rangle$ is a Sylow $3$-subgroup of $\bar X$,
its derived group is $\langle d_2,d_3\rangle$ and centre is $\langle d_3\rangle$. Notice that $\pi$ centralizes
$\langle d_1d_1^\pi, d_3\rangle$. We identify all these elements  and subgroups with their images in $X$.

\begin{lemma}\label{c1} The following hold:
\begin{enumerate}
\item $C_V(d_1) = [V,d_1]=\langle e\otimes x\otimes y\mid x,y \in \{e,f\}\rangle$ has dimension 4 and is totally isotropic;
 \item $C_V(d_2) = \langle e\otimes e\otimes x, (e \otimes f+ f\otimes e)\otimes x \mid x \in \{e,f\}\rangle$,
 $[V,d_2]=\langle e\otimes e\otimes x,(e \otimes f- f\otimes e)\otimes x  \mid x \in \{e,f\}\rangle$ have dimension $4$ and
 $[V,d_2,d_2]= \langle e\otimes e\otimes x\mid x \in \{e,f\}\rangle$ has dimension $2$. In particular, $d_2$ is not quadratic and $C_V(d_2)$ is not totally
 isotropic;
\item $C_V(d_3)= \langle e \otimes e \otimes e, f \otimes e \otimes e-e \otimes f \otimes e,
f \otimes e \otimes e-e \otimes e \otimes f\rangle$ has dimension $3$,
\begin{eqnarray*}[V,d_3]&=& \langle e \otimes e \otimes e,
 e\otimes f\otimes f+f\otimes f\otimes e+ f\otimes e\otimes f, \\&&
 f \otimes e \otimes e+e \otimes f \otimes e,
f \otimes e \otimes e+e \otimes e \otimes f,
 e \otimes f\otimes e+e \otimes e\otimes f\rangle\end{eqnarray*} has dimension $5$ and
$[V,d_3,d_3]= \langle e \otimes e \otimes e,
 e\otimes e\otimes f+e\otimes f\otimes e+ f\otimes e\otimes e
\rangle$ has dimension $2$.
In particular,  $d_3$ is not quadratic,  $C_V(d_3)< [V,d_3]$ and $C_V(d_3)$ is totally isotropic;
\item $C_V(\langle d_1,d_2\rangle)= \langle e\otimes e \otimes e, e\otimes e\otimes f\rangle$ has dimension $2$;
\item $C_V(\langle d_1,d_3\rangle)=\langle e\otimes e\otimes e, e\otimes e \otimes f -e\otimes f \otimes
e\rangle$ has dimension $2$;
    \item $C_V(\langle d_2, d_3\rangle)= \langle e\otimes e \otimes e, f\otimes e \otimes e+e\otimes f \otimes e+e\otimes e \otimes f\rangle$ has dimension
    $2$;
 \item $C_V(D) = \langle e\otimes e \otimes e\rangle$ has dimension 1; and  \item $C_V(\pi)=\langle x \otimes x \otimes y, (e\otimes  f + f\otimes e)\otimes y\mid x, y \in \{e,f\} \rangle$ has dimension 6.
\end{enumerate}
\end{lemma}

\begin{proof} This is a straight-forward calculation.  Note,  before  starting, we know that $\dim V/C_V(x) =\dim [V,x]$ for any $x \in X$. Furthermore, commutator spaces are generated by commutators with basis elements. So first determine $[V,x]$ and then verify that the spaces claimed to be $C_V(x)$ are indeed centralized. To see that $C_V(d_2)$ is not isotropic we calculate,
\begin{eqnarray*}
&&(e\otimes f\otimes e+f\otimes e\otimes e,e\otimes f\otimes f+f\otimes e\otimes f)\\
 &=& (e\otimes f\otimes e, e\otimes f\otimes f)+ (e\otimes f\otimes e,f\otimes e\otimes f) \\&&+ (f\otimes e\otimes e,e\otimes f\otimes f)+ (f\otimes e\otimes e, f\otimes e\otimes f)\\
&=& 0-1-1+0=1.
\end{eqnarray*}
To show that $C_V(d_3)$ is isotropic, just note that each vector in the basis of $C_V(d_3)$  is a sum  of tensors involving  $e$ at least twice.
Hence the form will involve $(e,e)=0$ in each product within the sum.
\end{proof}

\begin{lemma}\label{c2} The following hold:
\begin{enumerate}
\item $C_{ X}(\pi)$ is the image of the subgroup $$\langle (\pm b,b,I), (I,I,c), \pi \mid b \in \GSp(W), c \in \Sp(W)\rangle$$ of $\ov X$;
\item $|C_{X}(\pi)|= 2^7\cdot 3^2$, $C_R(\pi) \cong 2^2 \times \Q_8$ and $C_X(\pi)/C_R(\pi)\langle\pi\rangle \cong \Sym(3) \times 3$; and
\item $[R,\pi]$ is elementary abelian of order $2^3$ and is  the image of
$$\langle (b,b^{-1},I) \mid b \in O_2(\GSp(W))\rangle \ge \langle
 (I,-I,I), (-I,I,I) \rangle
.$$
\end{enumerate}
Furthermore $C_R(\pi) \ge [R,\pi]$.
\end{lemma}

\begin{proof} For part (i) we may focus on the base of  group $\bar X$. We calculate $(x,y,z)\in \bar X $, $x,y,z \in \GSp(W)$ is centralized by $\pi$ mod $K$ if and only if $xy^{-1} = \pm I$. The requirement that the element is in $\bar X$ rather that $\bar X^*$ means that $z \in \Sp(W)$.

For part (ii) we already see that  $|C_{X}(\pi)|= 2^7\cdot 3^2$ from (i). Also $C_R(\pi)$ is the image of
$$ \langle (\pm b,b,I), (I,I,c) \mid b,c
\in O_2(\GSp(W))\rangle \cong 2\times \Q_8 \times \Q_8$$  in $X$ and so $C_R(\pi)  \cong \Q_8
\times 2^2.$  Finally, letting $m$ be the element of $\GSp(W)$ that maps $e$ to $-e$ and fixes $f$, we have $ (m,-m, I)$ projects to an element of $C_X(\pi)$, centralizes $(I,I,d)$ and inverts $(d,d,I)$. Therefore $C_X(\pi)/C_R(\pi)\langle \pi \rangle \cong \Sym(3)\times 3$ and
this finishes the proof of  (ii).

For part (iii), we have $[O_2(\ov X),\pi] = \langle (b,b^{-1},I) \mid b \in O_2(\GSp(W))\rangle$ and, calculating $(c,c^{-1},I)(d,d^{-1},I)((cd)^{-1},cd,I)$ for non-commuting $c,d \in O_2(\GSp(W))$, shows that $(I,-I,I) \in [O_2(\ov X),\pi]$. Hence $(-I,I,I) \in [O_2(\ov X),\pi]$ as well.  Since, for elements of order $4$ in $O_2(\GSp(W))$, $b^{-1}= -b$, it is clear that  $[R,\pi ]\le C_R(\pi)$ and $[R,\pi]$ is elementary abelian of order $2^3$.
\end{proof}

Let $X^*$ be the image of $\ov X^*$ in $\GSp_8(3)$.

\begin{lemma}\label{out} Every coset of $O^2(C_{X^*}(\pi))\langle \pi \rangle$ in $C_{X^*}(\pi)$ contains an involution.
\end{lemma}

\begin{proof} Let $m$ be the element of $\GSp(W)$ which negates $e$ and fixes $f$. Then the image of the subgroup
$ \langle (m,m,I), (I,I,m)\rangle$  of $\ov X^*$  has order $4$, centralizes $\pi$ and complements $O^2(C_{X^*}(\pi))\langle \pi \rangle$. This
proves the result.
\end{proof}

By \cite[Proposition 4.6.9]{KL}, in $\Sp_8(3)=\Sp(V)$, the normalizer $L$ of $R$ has shape
$2^{1+6}_-.\Omega_6^-(2)$  and is uniquely determined up to conjugacy. Since $X$ normalizes $R$,  $X \le L$.
As $\Omega_6^-(2) \cong \PSp_4(3)$ has Sylow $3$-subgroups of order $3^4$, we see that $\pi$ is contained in both the subgroup
$X$ defined above and the companion ($3$-local) parabolic subgroup $P \ge RD\langle \tau,\pi\rangle$ with $P/R
\cong 3^{1+2}_+.\SL_2(3)$. Furthermore, in $P$, $\pi$ projects to the central element of $P/O_{2,3}(P) \cong
\SL_2(3)$.  Let  $Q $ be extraspecial of order $3^9$ and exponent $3$ and define $U$ and $M$ to be the
semidirect products \begin{eqnarray*} U&=&QL \approx 3^{1+8}_+{:}2^{1+6}_-.\Omega_6^-(2); \text { and}\\ M&=& QP
\approx 3^{1+8}_+{:}2^{1+6}_-.3^{1+2}_+.\SL_2(3).\end{eqnarray*}
We remark  that, as $\Aut(3^{1+8}_+)$ contains a subgroup isomorphic to $\GSp_8(3)$, these semidirect products exist. A similar construction does not work for extraspecial $2$-groups as  in general their automorphism groups do not contain the corresponding orthogonal groups. In fact $U$ is isomorphic to the centralizer of a $3$-central element in $\F_2$.

\section{The structure of  $3$-centralizers of type $\F_2$ and $\M(23)$}

Suppose that $G$ is a group, $S \in \syl_3(G)$, $Z= Z(S)$,  $H= C_G(Z)$ is similar to a $3$-centralizer in either
$\M(23)$ or $\F_2$. Hence $H$ is  a group with the same shape as $U$ or $M$ from  Section~3.
Set $Q= F^*(H)=O_3(H)\cong 3^{1+8}_+$. Since    $O_{3,2}(H)/Q \cong 2^{1+6}_-$, we see that $H/Z$ is a semidirect
product and is uniquely determined as a subgroup of the normalizer in $\Sp_8(3)$ of a subgroup isomorphic to
$2^{1+6}_-$.

Furthermore, if $H$ is similar to a $3$-centralizer in $\F_2$, then $H$ is a semidirect product and hence is
uniquely determined. In particular, if $H$ is similar to a $3$-centralizer in $\F_2$, then we identify $H$
with the group $U$ and if $H$ is similar to a $3$-centralizer in $\M(23)$, then we identify $H/Z$ with the group
$M/Z$. Note that in this latter case, the only difference between $M$ and $H$ is the elements from $M\setminus
M'$ which have order $3$ may correspond to elements which cube to $3$-central elements in $H$. In particular, in
this case we have $H' \cong M'$.  We also adopt all the notation established in Section 3 for elements of $H$
acting on $Q/Z$. When we want to distinguish between $U$, $M$ and $H$ we shall do so in the statements of the lemmas.

We let $F \le H$ be such that $F \cap Q = Z$ and $H=QF$. Thus $F/Z$ is a complement to $Q/Z$ in $H/Z$ and $F/Z \cong L$ if $H$ is similar to a $3$-centralizer of type $\F_2$ and $F/Z \cong P$ if  $H$ is similar to a $3$-centralizer of type $\M(23)$.  We now identify elements of $L$ and $P$ with elements of $F/Z$ and select preimages in $F$ for these elements of minimal order. Thus $F$ contain elements  which we denote by $ d_1, d_2, d_3, (1,2,3)$ for example and modulo $Z$ they behave just as the corresponding elements of $L$ or $P$ and furthermore they act on $Q/Z$ in exactly the way described. Notice that we know that $d_2, d_3$ and $(1,2,3)$ can be chosen of order $3$ and $d_1$ may have order $9$    if  $H$ is similar to a $3$-centralizer of type $\M(23)$.

We  define $$D = \langle d_1,d_2,d_3\rangle$$ and $$T = \langle d_2,d_3,(1,2,3)\rangle.$$  We may assume that $T$ is extraspecial of order $27$ and that $D$ has order $3^3$   if  $H$ is similar to a $3$-centralizer of type $\F_2$. We also assume that $ S \cap F = DT \in \syl_3(F)$.
We also let $\pi \in F$ correspond to  the involution as defined in Section 3.

The following lemma lists basic properties of $H$
and  interprets the information about the action of $L$ on the symplectic space $V$ collected in Lemma~\ref{c1}
in the situation when $L$ operates on an extraspecial group. The relationship between symplectic spaces and extraspecial groups  is well known and can be found in
Huppert \cite[III(13.7)]{Huppert}.

\begin{lemma}\label{basics} We have \begin{enumerate}\item  $N_G(S) \le H$ and  $S\in \syl_3(G)$.\item The $3$-rank of $H/Q$ is
$3$.
\item If $xQ$ is an element of order $3$ and $H$ is a $3$-centralizer in $\F_2$, then $xQ$ is conjugate to one of
 $d_1Q$, $d_2Q$ or $d_3Q$ in $H/Q$.
\item If $xQ$ is conjugate to $d_1Q$, then $C_{Q/Z}(x)$ has order $3^4$, its preimage is elementary abelian of order $3^5$ and
 $C_{Q/Z}(x) = [Q,x]/Z$.
\item If $xQ$ is conjugate to $d_2Q$, then $C_{Q/Z}(x)$ has order $3^4$ and its preimage is non-abelian of order $3^5$  with centre of order $3^3$.
\item If $xQ$ is conjugate to $d_3Q$, then $C_{Q/Z}(x) \le [Q,x]/Z$ and $C_{Q/Z}(x)$ has order $3^3$ and its preimage is elementary abelian of order $3^4$.
\item  $C_{Q/Z}(DQ/Q)$ has order $3$ and is centralized by $\pi$ and $|C_{Q/Z}(TQ/Q)| = 3^2$.
\item If $A$ is a subgroup of $S/Q$ of exponent $3$ and order $27$, then either $A= DQ/Q$ or $A= TQ/Q$.
\item $C_Q(\pi)$ is extraspecial of order $3^7$.
\item Let $\langle \sigma \rangle = (O_{3,2}(H)/Q)^\prime$ and $\wt \pi$ be an involution in $(O_{3,2}(H)/Q)\pi$,
then $\wt \pi$ is $H$-conjugate to either $\pi$ or $\sigma \pi$ in $O_{3,2}(H)\langle \pi\rangle$.
\end{enumerate}
\end{lemma}

\begin{proof} Parts (i) and (ii) are clear. Parts (iii) to (ix) correspond to the  statements in Lemma \ref{c1}.
The assertion in (x) is Lemma \ref{c2}.
\end{proof}

\begin{lemma}\label{el9} Every elementary abelian subgroup of order $9$ in $H/Q$ is $H/Q$-conjugate  to a subgroup of $DQ/Q$.
\end{lemma}

\begin{proof} The Sylow 3-subgroup of $H/Q$ is isomorphic to the wreath product $3 \wr 3$. Hence all the cyclic subgroups of order 3
in $S/Q $ not contained in $DQ/Q$ are conjugate under the action of $DQ/Q$. Let $E$ be an elementary abelian
group of order 9 in $S/Q$ with $E \not\leq DQ/Q$. Then $|C_{DQ/Q}(E)| = 3$ and so up to conjugacy $E$ is
uniquely determined and is contained in $ TQ/Q$.  As $N_{H/RQ}(TRQ/RQ) \cong 3^{1+2}{:}\SL_2(3)$ acts transitively on elementary abelian subgroups of order
$9$ in $TRQ/RQ$, we see that $E$ is conjugate into $DQ/Q$.
\end{proof}

\begin{lemma}\label{cent9} Let $E$ be an elementary abelian subgroup in $H/Q$ of order 9. Then $|C_{Q/Z}(E)| \leq 9$.
\end{lemma}

\begin{proof} By Lemma \ref{el9} we may assume that $E \leq DQ/Q$.  Now suppose $|C_{Q/Z}(E)| \geq 27$.  As $D$ is generated by elements which act quadratically on $Q/Z$, we see that
$DQ/Q = \langle E, t \rangle Q/Q$, where $t$ acts quadratically on $Q/Z$.  Therefore
$$|C_{Q/Z}(D)|=|C_{Q/Z}(E) \cap C_{Q/Z}(t)| \geq 9,$$ which contradicts Lemma \ref{basics}(vii).
\end{proof}

\begin{lemma}\label{invscz3}  $C_{H/Q}(d_3Q)$ has exactly three conjugacy classes of
involutions with representatives $\pi Q, \pi \sigma Q$ and $\sigma Q$ where $\sigma Q$
is the $2$-central element of $H/Q$. We have $|C_Q(\pi)|=3^7$, $|C_Q(\pi \sigma) |= 3^3$ and $|C_Q(\sigma)|=3$.
\end{lemma}

\begin{proof} Suppose first that  $H= QL$.  Since $Q$ is a $3$-group,
it suffices to work in $L$. We have $C_R(d_3)=\langle\sigma \rangle$  as by definition $d_3$ acts fixed point
freely on $R/\langle \sigma \rangle$. Furthermore, $C_{L/R}(d_3) $ has shape $3^{1+2}_+.\SL_2(3)$, and thus the
Sylow $2$-subgroup of $C_L(d_3)$ has shape $2.\Q_8$. It follows that there are at most $3$-conjugacy classes of
involutions in $C_L(d_3)$. Since $\pi$, $\pi \sigma$ and $\sigma$ all centralize $d_3$, we have the statement in
this case.  In the situation that $H'\cong M'$, we have the result by intersection with $L$. Finally the orders
of the centralizers of these involutions in $Q$ follow  from Lemma~\ref{basics} (ix) as $\sigma$ inverts $Q/Z$.
\end{proof}

We say that an element $e$ of $H$ acts as an element of \emph{type} $d_3Q$ provided that, viewed as an
element of $U$, $eQ$ is $U$-conjugate to $d_3Q$. Notice that this means that if $H$ has  a $3$-centralizer of
type $\M(23)$ then $eQ$ and  $d_3Q$ might not be $H$-conjugate.

\begin{lemma}\label{piact} Suppose that $y \in H$ has type $d_3Q$ with $[y,\pi]=1$. Then the following
hold.
\begin{enumerate}
\item  $|[C_{Q/Z}(y),\pi]|=3$;
\item if $A$ is a subgroup of order 9 in $DQ/Q$ such that $A = \langle yQ, B \rangle$, where $B$ is inverted by $\pi$, then $[C_{Q/Z}(A),\pi]=1$; and
\item $[C_{Q/Z}(T),\pi] = 1$.
\end{enumerate}
\end{lemma}

\begin{proof} It is enough to prove the statements for $H$ of type $\F_2$. Then, by definition, there is some $h \in H$
such that $(yQ)^h = d_3Q$. Furthermore $\pi^h \in C_H(d_3Q)$. By Lemma~\ref{invscz3} we may assume that
$\pi^h = \pi$ or $\pi \sigma$. By Lemma~\ref{basics}(ix)  $|C_Q(\pi \sigma)| = 3^3$, while
$|C_G(\pi^h)| = 3^7$. Hence we have $\pi^h = \pi$.

 Now (i) follows, as  $|[C_{Q/Z}(d_3),\pi]|=3$ by Lemma~\ref{c1}(iii).

We now may assume that $yQ = d_3Q$. As $A$ is not centralized by $\pi$ we have that $A \leq
O_3(C_{H/Q}(d_3Q)) = T Q/Q$. Hence $A = (D \cap T)Q/Q = \langle d_3Q,d_2 Q\rangle$ and so (ii) follows
from Lemma~\ref{c1}(vi).

As $\langle d_3,d_2 \rangle Q/Q\leq TQ/Q$, we may apply (ii) to get (iii).\end{proof}

We now determine the centralizer in $H$ of $\pi$.

\begin{lemma}\label{c3}  The following hold.
\begin{enumerate}
\item If $H$ is similar to a $3$-centralizer in $\F_2$, then $C_{H}(\pi)/\langle \pi \rangle $  has order $2^{10}\cdot 3^9$ and is similar to a $3$-centralizer in  ${}^2\E_6(2)$. Furthermore, $\pi\in C_U(\pi)'$ and $O_{3,2}(C_H(\pi)) \ge \langle \pi \rangle C_Q(\pi)C_R(\pi)$ which has shape $2\times 3^{1+6}_+{:}(2^2\times \Q_8)$.
\item If $H$ is similar to a $3$-centralizer in $\M(23)$,  then $C_{H}(\pi)/\langle \pi \rangle $ has order $2^{7}\cdot 3^9$  and is similar to a $3$-centralizer in $\M(22)$.
\end{enumerate}
\end{lemma}

\begin{proof} We first consider (i). In this case we may calculate in $U$. We have $C_U(\pi) = C_Q(\pi) C_L(\pi)$.  From Lemma~\ref{basics} (ix) we obtain
 $C_{Q}(\pi) \cong 3^{1+6}_+$. So we just need to determine $C_L(\pi)$. We  know
$C_L(\pi) \ge C_X(\pi)$  and  $$C_X(\pi) \approx 2\times ((2^2 \times \Q_8). (3\times \Sym(3)))$$  by Lemma~\ref{c2}(ii).

Since $\Omega_6^-(2)$ has a unique conjugacy class of involutions {whose centralizer has order divisible by 9},
we have  $$C_{L/O_2(L)}(\pi O_2(L)) \approx  2^{1+4}_+.(3 \times \Sym(3)).$$ Let $J$ be the preimage of this
group. Of course $C_L(\pi) = C_J(\pi)$. Using the fact that $C_X(\pi)/O_2(C_X(\pi)) \cong 3 \times \Sym(3)$, we
have $J= O_2(J)C_X(\pi)$. Furthermore, we know   $O_2(J)/R\langle \pi\rangle$ is a $J$-chief factor. In
particular, $J/R$ has no quotients isomorphic to $\Sym(4)$ or $\Alt(4)$.
 We claim  $C_J(\pi) R = J$.   By the definition of $J$, we have $R \langle \pi \rangle$ is
normal in $J$.  Thus $C_R(\pi) $, which, by Lemma~\ref{c2}, is  the preimage of  $Z(R\langle \pi \rangle/Z(R))$
is normalized by $J$.  Since $C_R(\pi) \ge [R,\pi]$ by Lemma~\ref{c2}, $R\langle \pi\rangle/C_R(\pi)$ is elementary abelian of order $2^3$. Hence $J/C_J(
R\langle \pi\rangle/C_R(\pi))$ is isomorphic to a subgroup of $\GL_3(2)$. Moreover,  as $J$ normalizes $R$, $J/C_J(
R\langle \pi\rangle/C_R(\pi))$ is isomorphic to a subgroup of $\Sym(4)$. Since $J/R$ has no quotients isomorphic to $\Sym(4)$ or $\Alt(4)$ and $\pi$ is centralized
by a Sylow $3$-subgroup of $J$ while $R/C_R(\pi)$ is not, we deduce that $J/C_J(
R\langle \pi\rangle/C_R(\pi)) \cong \Sym(3)$ or has order $3$. In particular,  $C_R(\pi)\langle \pi \rangle$ is normalized by $J$.  Therefore, by Lemma~\ref{perp},  $Z(C_R(\pi)\langle \pi \rangle)=
[R,\pi]\langle \pi \rangle$ and $[R,\pi]= R\cap [R,\pi]\langle \pi\rangle$ are also normalized by $J$. As
$|[R,\pi]|=2^3$, by Lemma~\ref{c2}(iii), we now have $C_J(\pi)$ has index at most $2^3$ in $J$.  Suppose that the
index is $2^3$.  Then, as $|R:C_R(\pi)|= 2^2$, we have $C_J(\pi) R$ has index $2$ in $J$. Since
$C_X(\pi)O^{2}(J)= J$, we  have a contradiction. Thus $|J:C_J(\pi)|= |R:C_R(\pi)|=2^2$ and so $C_J(\pi)R = J$ as claimed. In particular, we have
$|O_2(C_J(\pi))| = 2^{10}$ and $[R,\pi] \le Z(O_2(C_J(\pi)))$. We need to show that $O_2(C_J(\pi))/\langle
\pi\rangle \cong \Q_8 \times \Q_8\times \Q_8$.

 Of course $C_J(\pi)$ acts on $C_{Q}(\pi) \cong 3^{1+6}_+$. Let $K_0$ be the kernel of this action.
 Then $K_0$ acts faithfully on $[J,\pi]$ so $K_0$ is isomorphic to a subgroup of $\Sp_2(3)$. By Lemma~\ref{c1} (i), (ii) and (iii), $K_0$ is a $2$-group  and
 $K_0$ is normal in $C_L(\pi)$. Since $O_2(C_L(\pi))R/\langle \pi\rangle R$ is a $C_L(\pi)$-chief factor, we have that
 $K_0 \leq R\langle \pi\rangle$. So, as $K_0$ is isomorphic to a subgroup of $\Q_8$,  we have $K_0=\langle \pi \rangle$.
Therefore $\wt J =C_J(\pi)/\langle \pi\rangle$ is isomorphic to a subgroup of
$\Sp_6(3)$.

Note that  $\wt {[R,\pi]}$ is normalized by $\wt J$ and  has order $2^3$ by Lemma~\ref{c2} (iii). There is a
unique conjugacy class of elementary abelian subgroups of order $2^3$ in $\GSp_6(3)$ and the normalizer of such a subgroup is a
subgroup of $\Sp_6(3)$ which preserves the decomposition of the natural $6$-dimensional symplectic space into a
perpendicular sum of three $2$-dimensional subspaces and is isomorphic to   $W=\Sp_2(3)\wr \Sym(3)$. Let
$T_0$ be a Sylow 3-subgroup of $C_J(\pi)$. Since $\wt {T_0}$ acts non-trivially on {$\wt{ [R,\pi]}$}, we see that $T_0$
permutes the direct factors of the  base group $B$ of $W$ and $T_0\cap B$ is a diagonal element of order $3$.  Therefore  $O_2( \wt J) \le O_2(W)$. But this means that $O_2(\wt J)=O_2(W)\cong \Q_8 \times \Q_8
\times \Q_8$. In particular, $\wt {C_Q(\pi) C_J(\pi)}$ is similar to a centralizer in a group of type
${}^2\E_6(2)$ and this completes the proof of  (i).

Suppose that $H$ is similar to a $3$-centralizer in $\M(23)$. Then $H/Z \cong M/Z$. Hence $C_{H/Z}(\pi Z) \cong C_{M/Z}(\pi Z)$ and $|C_H(\pi Z)|= |C_M(\pi Z)|$. Hence, as the definition of a $3$-centralizer of type $\M(22)$ only requires the determination of the action of $O_{3,2}(C_H(\pi))$ on $C_Q(\pi)$, we may work in $M$.  We have seen that $|\pi^J|=|\pi^R|=4$, hence also $|\pi ^{C_M(\pi)}|=4$  which means that $C_M(\pi)R= N_M(\langle \pi\rangle R)$.    Now we intersect $C_J(\pi)$ with $M$ and see that $O_2(C_M(\pi))$ has order $2^8$ and $ C_{\wt M}(\pi)$ has order $2^7$ and
  contains $\wt {[R,\pi]}$. Thus $C_{\wt M}(\pi)$ is similar to a centralizer in a group of type $\M(22)$, {which is (ii)}.

  Suppose that $\pi \not \in C_L(\pi)'$. Then, as $C_R(\pi)=[C_R(\pi),C_L(\pi)]$, $\pi \not \in (C_L(\pi)/C_R(\pi))'$.
  But $C_L(\pi)/C_R(\pi)\approx 2^{1+4}_+.(\Sym(3)\times 3)$ and $\pi C_R(\pi)$ is  contained in the centre of this group, so we have  a contradiction.
  Thus $\pi \in C_L(\pi)'\le C_U(\pi)'$ as claimed.

\end{proof}

When $H$ is similar to a $3$-centralizer in $\F_2$, we need to determine the centralizer of a further involution.

\begin{lemma}\label{CR} Suppose that $H$ is similar to a $3$-centralizer in $\F_2$ and let $r \in R\setminus Z(R)$ be an involution. Then the following hold.
\begin{enumerate}
\item $|C_Q(r)|=3^5$, $C_R(r) \cong 2\times 2^{1+4}_-$, $C_{L}(r)R/R \cong 2^{4}.\Alt(5)$ and
$|O_2(C_H(r)) |=2^5$;
\item $C_H(r)/O_2(C_H(r))$ is similar to a $3$-centralizer in $\Co_2$; and
\item  $O_2(C_H(r)) \cong 2^{1+4}_-$ and  is the unique maximal $C_Q(r)$ signalizer in $C_H(r)$.
\end{enumerate}
\end{lemma}

\begin{proof} We may suppose that $H=U$. The involutions in $R$ correspond to singular vectors for the action of $L/R\cong \Omega_6^-(2)$ on $R/Z(R)$. Since
$L$ acts transitively on such elements, we have that $r$ is uniquely determined up to $L$-conjugacy. As $r$ and $r\sigma$ are conjugate in $R$, the subgroup
$\langle r, \sigma\rangle$ acts on $Q$ with $\sigma$ inverting $Q/Z$ and $r$ and  $r\sigma$ both centralizing an
extraspecial subgroup of order $3^5$.  Since $r$ is an involution in $R$, we have $C_R(r) \cong 2 \times
2^{1+4}_-$ and, as $L$ is transitive on the involutions in $R$, we get $C_L(r)R/R= C_{L/R}(r) \approx
2^4.\Alt(5)$ as this is the stabilizer  in $\Omega_6^-(2)$ of a singular $1$-space. Now $C_R(r)/\langle r\rangle$
acts faithfully on $C_Q(r)$ and in $\Sp_4(3)$, the normalizer of such a subgroup has shape $2^{1+4}_-.\Alt(5)$.
We conclude that  $C_{C_U(r)}(C_Q(r))$ has order $2^5$ and this completes the proof of part (i).

Part (ii) follows from the details of the proof (i) and the definition of a $3$-centralizer of type $\Co_2$.

Since $O_2(C_U(r))$ acts faithfully on $[Q,r]$ and is normalized by $C_L(r)$, we deduce that $O_2(C_U(r)) \cong
2^{1+4}_-$.  The final part of (iii) is clear as $C_Q(r)$ signalizers in $C_U(r)$ must centralize $C_Q(r)$.
\end{proof}

\begin{lemma}\label{CR2} Suppose that $H$ is similar to a $3$-centralizer in $\F_2$. Then the following statements hold.
\begin{enumerate}
\item  There are exactly two  $H$-conjugacy classes of elements in $Q\setminus Z$.
\item Let $r$ be an  involution in $R \setminus Z(R)$.  If $\rho \in C_{Q}(r)\setminus Z(Q)$, then $C_H(\rho)/\langle \rho\rangle$ is  similar to a
$3-$centralizer in $\M(22)$ and has order $2^8\cdot 3^{9}$. Furthermore, $C_H(\rho)$ splits over $\langle
\rho\rangle$.
\end{enumerate}
\end{lemma}

\begin{proof} We may suppose that $H = U$.
Let $E$ be an elementary abelian subgroup of $R$ of order $2^3$.
 Then, as $R \cong 2^{1+6}_-$,   we have $E/\langle \sigma\rangle$ corresponds to a maximal singular
subspace of $R/\langle \sigma \rangle$ and so, setting  $N=N_L(E) $, we have $N/R \approx 2^{1+4}_+.(3\times
\Sym(3))$ is the stabilizer in $ \Omega_6^-(2)$ of such a subspace and is a maximal subgroup of $L$.  Notice that $R$ acts transitively by conjugation on the set $\mathcal H$ of hyperplanes of $E$ which complement $\langle \sigma\rangle$.  Since $[Q,E]=[Q,\sigma]=Q$, $Q/Z$ decomposes as a direct sum of centralizers of elements of $\mathcal H$.  Let $E_0 \in \mathcal H$ and set $Q_0 = C_Q(E_0)$.  Noting that $|\mathcal H|=4$, we deduce that $Q_0$ is extraspecial of order $3^3$.
 Since $R$ acts transitively on both $\mathcal H$  and $C_Q(E_0)^R$ and, since $N$ acts on both sets, we have $N= N_N(E_0)R = N_N(Q_0)R$ and  $N_N(Q_0) = N_N(E_0)$.
Recalling that  $NQ$ is a maximal subgroup of $U$, we obtain $N_U(Q_0) = N_N(E_0)Q$.  Finally, as an element of order $3$ from $N$ acts non-trivially on $N_R(E_0)/E$, we infer that  $N_U(Q_0)/C_U(Q_0)Q_0 \cong \Sp_2(3)$.

Let $\rho \in Q_0$. Then, as $C_R(E_0)Q_0$ acts transitively on the non-central elements of $Q_0$,  $C_{N_U(Q_0)}(\rho) R= N_L(Q_0)R$  and  $C_R(\rho) = E_0$. Thus $$C_L(\rho) \approx
2^2.2^{1+4}_+.(\Sym(3) \times 3).$$  In particular, $|C_L(\rho)|= 2^8\cdot 3^2$ and $|C_U(\rho)|=2^8\cdot 3^{10}$
 and so there are exactly  $1440$ conjugates of $\rho Z$ in $Q/Z$. Now note that the $1$-space
$\langle e\otimes e\otimes e\rangle$ of $V \cong Q/Z$  is centralized by $$T_1=\langle \sigma,  D, (m,m,I),
(I,m,m), (1,2,3),(1,2)\rangle \le L$$ where $m \in \GSp(W)$ fixes $e$ and negates $f$ and that $C_R(e\otimes e\otimes
e)= 1$. Since $T_1 R$ has index $2^{10}\cdot 5$ in $L$, there are at least  $5120$ conjugates of $ e\otimes
e\otimes e$ in $Q/Z$. As $\rho Z$ and $e\otimes e \otimes e$ have different $3$-parts in the orders of their centralizers  and $3^8-1 = 1440+ 5120$, we have proved that there are exactly two $U$-conjugacy classes of elements in
$Q\setminus Z$. This completes the proof of (i).

Now define $J=O_2(C_L(\rho))$  and note that $J$ has order $2^7$. {We have that $C_Q(J) \cap [Q,J] = Z$, so $C_Q(\rho) = \langle \rho \rangle \times [Q,J]$ and then}
$[Q,J]$ is extraspecial of order $3^7$. Furthermore, $J \cap R = E_0 \le Z(J)$ and $E_0$ decomposes $[Q,J]/Z$ into a
perpendicular  sum of three $2$-dimensional subspaces each centralized by an involution of $E_0$. It follows that
$C_L(\rho) $ embeds into the subgroup $(\Sp_2(3) \wr \Sym(3)).2$ of $\GSp_8(3)$. Since $J$ centralizes each
involution in $E_0$ and is normalized by a Sylow $3$-subgroup, we have $J$ embeds into $\Q_8\times \Q_8\times
\Q_8$ and it follows that $[Q,J]C_L(\rho)$ is similar to a $3$-centralizer in $\M(22)$. Finally, as  {$C_U(\rho)=
C_Q(\rho)C_L(\rho) = \langle \rho \rangle \times [Q,J]C_L(\rho)$}, $C_U(\rho)$ splits over $\langle \rho\rangle$. Thus (ii) holds.
\end{proof}

\section{The centralizer in $G$ of $\pi$ and the proof of Theorem~\ref{Main1}}

We continue the notation of the previous sections. Set $K= C_G(\pi)$ and denote by  $\;\wt{}\;:K \rightarrow K/\langle \pi\rangle$  the natural
homomorphism from $K $ to $K/\langle \pi\rangle$. In this section we prove Theorems~\ref{tpi1} and
\ref{tpi2}. Using Theorem~\ref{tpi1} we then prove Theorem~\ref{Main1}.

\begin{theorem}\label{tpi1} If $H$ is similar to a $3$-centralizer in $\M(23)$, then
$K/\langle \pi\rangle\cong  \M(22)$.
\end{theorem}

\begin{theorem}\label{tpi2}
If $H$ is similar to a $3$-centralizer in $\F_2$, then $K \cong (2\udot {}^2\E_6(2)){:}2$.
\end{theorem}

We prove these theorems simultaneously through a series of lemmas. Note that by Lemma~\ref{c3},
$\wt{C_H(\pi)}$ is similar to a $3$-centralizer in $\M(22)$ when $H$ is similar to a $3$-centralizer of type
$\M(23)$ and $\wt{C_H(\pi)}$ is similar to a $3$-centralizer of type ${}^2\E_6(2)$ when $H$ is similar to a
$3$-centralizer of type $\F_2$. Therefore our goal in this section will almost be reached once we show that $\wt
Z$ is not weakly closed in $\wt{C_S(\pi)}$ with respect to $\wt{C_G(\pi)}$, for then we will apply the main
theorems of \cite{PSS} to $\wt{K}$.

\begin{lemma} We have $C_S(\pi) \in \Syl_3(K)$.
\end{lemma}

\begin{proof}  Since  $\wt {C_{H}(\pi)}$ is similar to a $3$-centralizer of type  $\M(22)$ or ${}^2\E_6(2)$, we have $Z(C_S(\pi)) = Z$.
Therefore $O^2(N_K(C_S(\pi))) \le C_K(Z) $ which then means that $C_{S}(\pi) \in \syl_3(K)$.
\end{proof}

\begin{lemma}\label{wk1} If $Z$ is not weakly closed in $Q$ with respect to $G$, then $Z$ is not weakly
closed in $C_S(\pi)$ with respect to $K$.
\end{lemma}

\begin{proof} Assume that $g \in G$ and $Y= Z^g \le Q$. Then $C_Q(Y) \le C_G(Y)=H^g$ and $C_Q(Y)$ contains
an extraspecial subgroup of order $3^7$.  Since, by Lemma~\ref{basics} (ii),  $H/Q$ has no extraspecial subgroups of order $3^7$, we have that
$Z \le Q^g$. Now $\Phi(Q^g \cap Q) \le \Phi(Q^g) \cap \Phi(Q)= Z\cap Y=1$ which means that $Q\cap Q^g$ is
elementary abelian. As $Q^g$ is extraspecial of order $3^9$, we get $|Q^g:Q^g \cap Q| \ge 3^4$. Since the
$3$-rank of $H/Q$ is $3$ by Lemma~\ref{basics} (ii), we have $(Q^g\cap H)Q/Q $ is elementary abelian of order $3^3$.
By Sylow's Theorem we may suppose that $(Q^g\cap H)Q \le S $ and  then Lemma~\ref{basics}(vii) gives $$ZY/Z =
C_{Q/Z}(Q^g \cap H)$$ and this group is centralized by $\pi$. Thus $\pi$ centralizes
$Y$ and so $\pi \in H^g$ and, as $C_{Q^g}(\pi) \ge Z$, we have $C_{Q^g}(\pi)$ is extraspecial. Now $\langle
C_{Q}(\pi), C_{Q^g}(\pi)\rangle \le K$ and this group normalizes $ZY$ and acts transitively on the cyclic
subgroups of $ZY$. This means that $Z$ is not weakly closed in $C_S(\pi)$ with respect to $K$ as claimed.
\end{proof}

The following lemma is the main technical result of this section.

\begin{lemma}\label{wk2} $Z$ is not weakly closed in $C_S(\pi)$ with respect to $K$.
\end{lemma}

\begin{proof} Assume that $Z$ is weakly closed in $C_S(\pi)$ with respect to $K$.
By Lemma~\ref{wk1}, $Z$ is weakly closed in $Q$ with respect to $G$. Since $Z$ is not weakly closed in $S$
with respect to $G$, there exists $g \in G$ and $Y=Z^g$ such that   $Y \le S$ and $Y \not\le Q$.  By Lemma
\ref{el9} we may additionally assume that $Y\le DQ\le S$. Note that $Y$ is weakly closed in $Q^g$. As $Z$ and $Y$ commute,  $Z \le C_G(Y)= H^g$,  $Z \not \le Q^g$ and we may
assume that $Z \le (DQ)^g$.

\begin{claim} \label{cl1} The following hold:
 \begin{enumerate}
\item $Q\cap Q^g$ is elementary abelian;
\item $|C_Q(Y)| \ge 3^3$ and $|C_{Q^g}(Z)|\ge 3^3$;
\item $Q\cap Q^g \neq 1$;
\item $Q\cap Q^g = C_Q(Y) \cap C_{Q^g}(Z);$
\item $|Q\cap Q^g\cap [Q,\pi]| \le 3$; and
\item $[Q,Y]\cap Q^g \neq 1$.
\end{enumerate}
\end{claim}
\bigskip

We have $$\Phi(Q\cap Q^g) \le \Phi(Q) \cap \Phi(Q^g) \le Z \cap Y=1$$ and so (i) holds.

 By
Lemma~\ref{basics}(iii), $YQ/Q$ acts as a $U$-conjugate of $\langle d_1\rangle Q$, $\langle d_2\rangle Q$, or $\langle d_3\rangle Q$ on $Q/Z$ and $Z$ acts in
a similar way on $Q^g/Y$ (though perhaps not as the same type of element). Hence part (ii) comes from
Lemma~\ref{basics}(iv), (v) and (vi) as $C_Q(Y)$ has index at most $3$ in the preimage of $C_{Q/Z}(Y)$.

 Assume  that $Q \cap Q^g=1$. Then $|C_Q(Y) Q^g/Q^g| \ge 3^3$ by (ii).   Since  $C_Q(Y)\cap Q^g=1$ and $C_Q(Y)Q^g/Q^g$ has exponent $3$, we have that $|C_Q(Y)|= 3^3$ and $|C_{Q/Z}(Y)|=3^3$ by Lemma~\ref{basics} (ii), (iii) and (iv). In particular, Lemma~\ref{basics} (vi) implies that $C_Q(Y)$ is elementary abelian.
  Now Lemma~\ref{basics} (vii) implies that $C_{Q^g/Y}(C_Q(Y))$ has order
$3$. Since $$[C_Q(Y), C_{Q^g}(Z)]\le Q \cap Q^g=1,$$ we have a contradiction as $C_{Q^g}(Z)$ has order at
least $3^3$ by (ii). Hence (iii) holds.

As  $[Q\cap Q^g,ZY]= 1$,  $$Q\cap Q^g \ge C_Q(Y) \cap C_{Q^g}(Z)\ge Q\cap Q^g.$$ This is (iv).

Part (v) follows as $[Q,\pi]\cong 3^{1+2}_+$ by Lemma~\ref{basics} (xi) and $Q\cap Q^g$ is elementary abelian by (i) and does not contain $Z$.

 Finally, suppose that $[Q,Y]\cap Q^g=1$. Then, as $C_Q(Y)Z/Z \le  C_{Q/Z}(Y)$, part (iv) implies that
$C_{Q/Z}(Y) \not \le [Q,Y]/Z$. Hence Lemma~\ref{basics} (iv), (v) and (vi)  imply  $YQ/Q$ is a $U$-conjugate of $\langle d_2\rangle Q$. Let $E$
be the preimage of $C_{Q/Z}(Y) \cap [Q,Y]/Z$. Since $E$ is the centre of the preimage of $C_{Q/Z}(Y)$,
Lemma~\ref{basics} (v) implies that $|E|= 3^3$. It follows from (i) and (iv) that  $E(Q\cap Q^g)$ is abelian and thus we
have $|Q\cap Q^g|= 3$ as the preimage of $C_{Q/Z}(Y)$ is non-abelian. Using (ii) and (iii), we now have
$|C_{Q^g}(Z)Q/Q|\ge 3^2$. Set $E_0= E \cap C_Q(Y)$ and note that $E_0\le E \le [Q,Y]$ and is $C_{Q^g}(Z)$-invariant. Therefore  $$[E_0,C_{Q^g}(Z)] \le [Q,Y]\cap Q^g=1.$$  If $E_0 = E$,
then $C_{Q/Z}(C_{Q^g}(Z)) \ge E(Q\cap Q^g)/Z$ has order at least $3^3$ and this contradicts Lemma~\ref{cent9}.
So  $E_0$ has index $3$ in $E$ and $C_{Q/Z}(Y) = EC_{Q}(Y)/Z$. Since $Q\cap Q^g$ is centralized by $C_Q(Y)$, we
now have $E(Q\cap Q^g)$ is centralized by $C_Q(Y)$ and this implies that the preimage of $C_{Q/Z}(Y)$ is abelian,
which is a contradiction. Thus (vi) holds. \qedc

\bigskip
\begin{claim}\label{cl2}  The following hold.
\begin{enumerate}
\item Every element of $[Q,Y]Y\setminus [Q,Y]$ is $Q$-conjugate to an element of $YZ \setminus Z$; and
\item $C_{Q}(Y)/Z= C_{Q/Z}(Y)$.
\end{enumerate}
\end{claim}
\bigskip

Clearly $Q$ normalizes both $[Q,Y]Y$ and $[Q,Y]$. Thus $Q$ permutes the elements of $[Q,Y]Y\setminus [Q,Y]$. Since
$$|(YZ/Z)^Q|=|Q/Z:C_{Q/Z}(Y)|= |[Q/Z,Y]|= |[Q,Y]/Z|,$$ the claim in (i) follows. Now, if $C_Q(Y)/Z<C_{Q/Z}(Y)$, then
the preceding calculation shows that all the elements of
$$[Q,Y]Y\setminus [Q,Y]$$ are conjugate to elements of $Y$. But then  because of \ref{cl1}(vi) we have $([Q,Y]\cap Q^g)Y$
contains conjugates of $Y$ other than $Y$ and this contradicts $Y$ being weakly closed in $Q^g$. \qedc

\bigskip
\begin{claim}\label{cl3} After conjugating $Y$ by a suitable element of $H$ we may assume  $Y \le K$.
\end{claim}
\bigskip

We first show that if $[\pi, Y] \leq QR$ then there is a $QR$-conjugate of $Y$, which is centralized by $\pi$. Under this assumption,
Lemma \ref{basics}(x) shows that we  may assume that $[Y,\pi] \leq Q$. We consider the subgroup $\langle \pi\rangle QY$. We
have that $\pi$ normalizes $[Q,QY]=[Q,Y]$.  Using Lemma~\ref{perp} and \ref{cl2} we have $C_Q([Q,Y]) =
C_Q(Y)$. So $\pi$ normalizes $C_Q(Y)$. Thus $\pi$ normalizes $C_{QY}(C_Q(Y)) =[Q,Y]Y$. By \ref{cl2} every
element of $[Q,Y]Y\setminus [Q,Y]$ is $Q$-conjugate to an element of $ZY$. Thus  some $Q$-conjugate of $ZY$ is
normalized by $\pi$. Hence we may assume that $ZY$ is normalized by $\pi$. Since $\pi$ centralizes $YQ/Q$ and
$Z$, $\pi$ centralizes $ZY$ and thus $\pi$ centralizes $Y$.

 So we may assume that $[\pi, Y] \not\leq QR$.  Since by Lemma~\ref{el9} every element of order $3$ in
$H/QR$ is conjugate into $DQR/QR$, we may assume that $YQR/QR \leq DQR/QR$.

 We have that $C_{DQ/Q}(\pi)= \langle d_1d_1^\pi,d_3\rangle Q/Q$ has order $3^2$. Set $T_0 =TQR$. Then $N_{H}(T_0)/T_0 \cong \SL_2(3)$.  The action of $N_H(T_0)$ on $T_0/QR$ shows that
elements of order three in $DQ/Q$ are either  $T_0Q/Q$-conjugate   into  $C_{DQ/Q}(\pi)$ or are contained in
$DQ/Q \cap TQ/Q$. Since the former possibility has been eliminated,  we may assume that $YQ/Q \leq (D\cap T)Q/Q$. Now using the conjugation action of $T_0/QR$ we may assume that
$YQR/QR$ is inverted by $\pi$. We will show that this is impossible.

By Lemma~\ref{basics}(viii) we may additionally assume that $YQ/Q$ acts as  $\langle d_2Q \rangle$ on $Q/Z$. Set $x = d_1d_1^\pi$. As  $|C_R(d_2)| = 2^3$, $|C_R(d_3)| = 2$ and
$|C_R(d_1)| = 2^5$, the elements $d_i$, $1\le i\le 3$, can be distinguished by their action on $R$. Hence we get that $|C_R(d_1d_1^\pi)| = 2^3$ and so $d_1d_1^\pi $ is conjugate in $U$ to $d_2$. Suppose that $H =
U$. Then we may conjugate $Y$ in $H$ to $d_1d_1^\pi$ and hence we may choose $Y$ with $[Y,\pi] \leq QR$ right from
the beginning. However we have already considered this case.

So for the final part of the proof of \ref{cl3} we may assume that $H/Z \cong  M/Z$. Then Lemma~\ref{basics}(v) and \ref{cl2} show that $C_Q(Y)$
(recall $Y$ is of type $d_2$) is a direct product of an elementary abelian group of order $3^2$ with an
extraspecial group of order $3^3$. In particular $C_Q(Y)^\prime = Z$. As $Z \not \le Q^g$,  Lemma~\ref{basics} (vi) shows that $C_Q(Y)Q^g/Q^g$ is extraspecial of order $3^3$ and so $Q \cap Q^g$ is elementary abelian of
order $3^2$. This yields $[\pi^g,Z] \leq Q^gR^g$. By the argument above there is a $Q^gR^g$-conjugate of
$\pi^g$, which is centralized by $Z$. So we may assume that $[Z,\pi^g] = 1$. In particular $\pi^g \in H$. Hence
$\pi^g \in \langle R, \pi \rangle$ as $H/Z \cong M/Z$. However $[\pi^g, Y] = 1$, which implies $\pi^g \in C_R(Y) \cong
\Q_8$.

So we have that $\pi^g$ acts as an element of $Z(R)$ on $Q$, in particular $C_Q(\pi^g) = Z$. As $\pi^g$ acts on $Q \cap Q^g$,
we now get that this group is inverted by $\pi^g$.  Since  $|Q\cap Q^g|=3^2$, this contradicts \ref{cl1} (v) and proves the claim.
  \qedc

As  $Z \not \le Q^g$ and $Z\le H^g$, we may apply \ref{cca} to this configuration. Notice that when we adjust $Z$, we also adjust $\pi$ and therefore we may still assume $Z \le K$.

\begin{claim}\label{cca} After conjugating $Z$ by an element of $H^g$, we may assume that $[Z,\pi^g]=1$.
\end{claim}\qedc

We now have $\langle \pi, \pi^g \rangle \le H \cap  H^g$ and we may assume that $\langle \pi ,\pi^g\rangle$ is a $2$-group.

\bigskip
\begin{claim}\label{cl4}  $\pi$ and $\pi^g$ are neither  $H^g$-conjugate nor  $H$-conjugate.
\end{claim}
\bigskip

 If $\pi^{gh}= \pi$  for some $h \in H^g$, then $Z^{gh}= Y^h=
Y$ which means that $Z$ is not weakly closed in $C_S(\pi) $ with respect to $K$. This contradicts our initial assumption  that $Z$ is   weakly closed in $C_S(\pi)$ with respect to $K$.
Therefore $\pi$ and $\pi^g$ are not $H^g$-conjugate.

Similarly, suppose that $\pi $ and $\pi^g$ are $H$-conjugate. Then there exists $t\in H$ such that $\pi^t = \pi^g$. Let $s\in G$ be such that $gs= t$. Then $\pi^{gs}= \pi^g$ and so $s \in C_G(\pi^g) $. However, $Y^{s} = Z^{gs} = Z^t = Z$.
 This shows that $Y$ is not weakly closed in  a Sylow$3$-subgroup of $G_G(\pi^g)$ with respect to $C_G(\pi^g)$.
 \qedc

\begin{claim}\label{cl6} $C_Q(Y)$ and $C_{Q^g}(Z)$ are  abelian.
\end{claim}
\bigskip

Assume  that $C_{Q}(Y)$ is non-abelian. Then $Y$ does not act quadratically on $Q$ and $|C_Q(Y)|=3^5$ by
Lemma~\ref{basics} (iv), (v) and (vi) and \ref{cl2} (ii). Since $C_Q(Y)'=Z$, $C_Q(Y)$ has exponent $3$ and $Z \not \le Q^g$,
$C_Q(Y)Q^g/Q^g$ is extraspecial of order $3^3$ by Lemma~\ref{basics} (viii). Thus $|C_Q(Y) \cap Q^g| = 3^2$. Further $Z$ acts as an element of
type $d_3$ on $Q^g$.  So by Lemma~\ref{basics} (vi) and \ref{cl2}(ii) $C_{Q^g}(Z)$ is elementary abelian of order
$3^4$.  By Lemma~\ref{basics} (vii) we have  $|C_{Q^g/Y}(C_Q(Y))|=3^2$ and by Lemma~\ref{piact} (iii)
$C_{Q^g}(C_Q(Y))/Y = (Q\cap Q^g)Y/Y$ is centralized by $\pi^g$.
 By \ref{cl4}, $\pi$ is not $H^g$-conjugate to $\pi^g$. Thus, using  Lemma~\ref{invscz3}, gives $C_{Q^g}(\langle \pi, \pi^g \rangle) = Y$ and so $(Q\cap Q^g)Y/Y$
is inverted by $\pi$. As $|Q\cap Q^g| = 3^2$ this again contradicts \ref{cl1} (v). Hence  $C_Q(Y)$ is elementary abelian.  A symmetric argument shows that $C_Q(Z)$ is elementary abelian.\qedc

%
%
%

\bigskip
\begin{claim}\label{cl5} At least one of the following hold.\begin{enumerate}
\item  $YQ/Q$ does not act fixed point freely on $RQ/Z(R)Q$; or
\item $ZQ^g/Q^g$
does not act fixed point freely on $R^gQ^g/Z(R^g)Q^g$.\end{enumerate}
\end{claim}
\bigskip

Suppose that (i) and (ii) both fail. Then the non-trivial elements of $Z$ and $Y$  act as  elements of type $d_3$ on $Q^g$ and $Q$ respectively. In particular  $C_Q(Y)$
and $C_{Q^g}(Z)$ are both elementary abelian of order $3^4$ by Lemma~\ref{basics} (vi) and \ref{cl2} (ii). By
Lemmas~\ref{basics} (vii) and \ref{el9}, we may suppose that $C_{Q^g}(Z)Q/Q  \le DQ/Q$.

By Lemma~\ref{invscz3},  $N_{H^g/Q^g}(Z)$ has three conjugacy classes of involutions and they are distinguished by their centralizer on $Q^g$.
Since $\pi$ is not $H^g$-conjugate to $\pi^g$ by \ref{cl4},  Lemma~\ref{basics}(ix) yields  $\pi$ either inverts $Q^g/Y$ or centralizes a subgroup of $Q^g/Y$ of order $3^2$.

 If $\pi$ inverts $Q^g/Y$, then $|C_{Q^g}(Z)Q/Q| \le 3^2$ as
$[DQ/Q,\pi]$ has order $3$ and $C_{Q^g}(Z)Q/Q \ge YQ/Q$. Thus $|C_{Q^g}(Z) \cap Q| \ge 3^2$ and so $|Q \cap Q^g|
\ge 3^2$ with $[Q\cap Q^g,\pi]= Q\cap Q^g$ which is against \ref{cl1} (v). Therefore $\pi$ does not invert $Q^g/Y$ and so $C_{Q^g}(\pi)$ is extraspecial of order $3^3$.
%
%
Then, as  $C_{Q^g}(Z)$ is abelian and $Z$ acts
on $C_{Q^g}(\pi)$, we get $|C_{Q^g}(\pi) \cap C_{Q^g}(Z)| = 3^2$. Thus $|[C_{Q^g}(Z),\pi]|= 3^2$. Hence, as
$|[C_{Q^g}(Z)Q/Q,\pi]|\le 3$, we obtain $|[Q\cap Q^g,\pi]|\ge 3$. Since $|[C_{Q}(Y),\pi]|=3$ as above, we infer
that  $$|[C_{Q}(Y),\pi]|=| [Q\cap Q^g,\pi]|= |[C_{Q^g}(Z)Q/Q,\pi]|=3.$$ It now follows that $C_Q(Y)Q^g/Q^g$ is
centralized by $\pi$ and this means that $|Q\cap Q^g|\ge 3^2$. Therefore $C_{Q^g}(Z) Q/Q= [C_{Q^g}(Z),\pi]YQ/Q$
and $|Q\cap Q^g| = 3^2$. As the non-trivial elements of $Y$ are of type $d_3$, we may apply Lemma~\ref{piact}(ii).   This shows that
$C_{Q}(C_{Q^g}(Z))$ is centralized by $\pi$. But $Q \cap Q^g \leq C_{Q}(C_{Q^g}(Z))$ and is not centralized by
$\pi$. This contradiction proves \ref{cl5}.\qedc
\bigskip

Finally, we can complete the proof of Lemma~\ref{wk2}. By \ref{cl6} and \ref{cl5} we may assume that $YQ/Q$ does
not act fixed point freely on $RQ/Z(R)Q$ and $C_Q(Y)$ is elementary abelian.
In particular, the non-trivial elements of $Y$ act on $Q/Z$ as elements  of type $d_1$  by Lemma~\ref{basics} (vi), (v) and (vi). Therefore,
by Lemma~\ref{basics}  (iv) and \ref{cl2},  $|C_Q(Y)|=3^5$ and $C_Q(Y) = [Q,Y]$.

Thus  $|C_Q(Y)Q^g/Q^g|\le 3^3$ and $|C_Q(Y)\cap Q^g|\ge 3^2$. If  $|C_Q(Y)Q^g/Q^g| = 3^3$ then by
Lemma~\ref{basics}(vii) we have that $|C_{Q^g/Y}(C_Q(Y))| = 3$, contradicting $|C_Q(Y) \cap Q^g| \ge 3^2$. Hence
we have that $|C_Q(Y)Q^g/Q^g|\le 3^2$ and $|C_Q(Y)\cap Q^g|\ge 3^3$. If   $|C_Q(Y)Q^g/Q^g| = 3^2$, then
Lemma~\ref{cent9} implies  $|C_{Q^g/Y}(C_Q(Y))| \le 3^2$, which is again a contradiction. Thus  $C_Q(Y)Q^g/Q^g =
ZQ^g/Q^g$ has order $3$ and  $|Q\cap Q^g|=  3^4$.
 In particular, $C_{Q^g}(Z)$ has order $3^5$ and so by Lemma~\ref{basics}  $Z$  also acts as an element of type $d_1$ on $Q^g$. Hence
 $C_{Q^g}(Z) = [Q^g,Z]$.

Putting $$W=(Q\cap Q^g)YZ=[Q,Y]Y=[Q^g,Z]Z,$$ we have $W$ is normalized by $I=\langle Q, Q^g\rangle$ and we consider
the action of $I$ on $W$.  We remark that $W$ is elementary abelian of order $3^6$. Since $Q^g$ does not centralize $Z$ we have $|Z^I| \neq 1$. In addition, by hypothesis
$(Q\cap Q^g)Z=C_Q(Y)=[Q,Y]$ contains  exactly one conjugate of $Z$ (namely $Z$) and $(Q\cap Q^g)YZ \setminus
[Q,Y]$ contains a multiple of 81 conjugates of $Z$ forming orbits under the action of $Q$. Since $(Q\cap Q^g)Y$
contains only one conjugate of $Z$, we infer that $|Z^I|= 82$ or $163$. Since neither $41$ nor $163$ divides
$|\GL_6(3)|$, we have a contradiction. This completes the proof of the lemma.
\end{proof}

We now prove the main theorems of this section.

\begin{proof}[The proof of Theorem~\ref{tpi1}] Suppose that $H $ is similar to a $3$-centralizer in $\M(23)$.
Then  Lemma~\ref{c3} (ii) says that  $C_H(\pi)/\langle \pi\rangle$ is
similar to a $3$-centralizer in $\M(22)$.  Furthermore, by Lemma~\ref{wk2}, $\wt Z$ is not weakly
closed in $\wt {C_{S}(\pi)}$ with respect to $\wt K$. Therefore \cite[Theorem 1]{PSS} implies that $F^*(\wt K)$
is isomorphic to $\M(22)$. Since,  by Lemma~\ref{c3} (ii), $|C_H(\pi)|= 2^8\cdot 3^9$, we have  $\wt K \cong \M(22)$.
\end{proof}

\begin{proof}[The proof of Theorem~\ref{tpi2}] Suppose that $H $ is similar to a $3$-centralizer in $\F_2$.
Then because of Lemmas~\ref{c3} (i) and \ref{wk2}, we may apply \cite[Theorem 2]{PSS} to get that $\wt K \cong
{}^2\E_6(2)$ or ${}^2\E_6(2).2$. Since Lemma~\ref{c3} (i) also states that $|\wt{C_{H}(\pi)}|=2^{10}\cdot 3^9$,
 we have $\wt K  \cong {}^2\E_6(2).2$. By Lemma~\ref{c3}, we have $\pi \in C_H(\pi)'$ and therefore, $F^*(K)
\cong 2\udot {}^2\E_6(2)$. Now $C_K(Z)$ contains  $C_{X^*}(\pi)$ from Section 3 and Lemma~\ref{out} implies there
is an involution in $K$ but not in $F^*(K)$. Hence $K \cong (2\udot{}^2 \E_6(2)){:}2$ as claimed.
\end{proof}

\begin{proof}[The proof of Theorem~\ref{Main1}] Suppose that  $C_G(Z)$ is similar to a $3$-centralizer in $\M(23)$.
Then $ \wt K \cong \M(22)$ by Theorem~\ref{tpi1}. To complete the proof that $G \cong \M(23)$, we need to show
that $\pi$ is not weakly closed in $K$ with respect to $G$. However, $\pi$ is not a $2$-central element in $H$
and so the required property is self evident.

\end{proof}

\section{A further $3$-centralizer}

Having proved Theorem~\ref{Main1}, for the remainder of the paper we suppose that $H$ is similar to a $3$-centralizer in $\F_2$.  In particular, $H \cong U$ where $U$ is as defined at the end of Section~3.  We continue the notation
established in the previous sections. In particular, $S\in \syl_3(G)$ is normalized by $\pi$ and  $K= C_G(\pi) \cong (2\udot{}^2\E_6(2)){:}2$ with $C_S(\pi)\in \Syl_3(K)$.

By Proposition~\ref{2e62facts} (iv), there exists  $\rho \in C_Q(\pi)\le K$  such that $$C_{\wt K}(\wt \rho) \cong 3 \times \PSU_6(2).2$$ and
$C_{C_S(\pi)}(\rho) \in \Syl_3(C_{K}(\rho))$.   Let $E_\rho= E(C_K(\rho))$.
Then $E_\rho \langle \pi \rangle /\langle \pi \rangle \cong \PSU_6(2)$ and $C_{\wt E_\rho}(\wt Z)$ has shape $3^{1+4}_+.(\Q_8\times
\Q_8).3$.

The only lemma in this section
is as follows.

\begin{lemma}\label{Crho}  $N_G(\langle \rho\rangle) \cong \Sym(3) \times \Aut(\M(22))$.
\end{lemma}

\begin{proof} Since $C_K(\rho)$ has a composition factor isomorphic to $\PSU_6(2)$ and $C_G(Z)$ does not, we know
that $\rho$ is not a $3$-central element of $G$. As $Z$ is conjugate to an element of $C_Q(\pi)\setminus Z$ by
Proposition~\ref{2e62facts} (iii), Lemma~\ref{CR2} (ii) implies that $C_{H}(\rho)/\langle \rho\rangle$ is similar to a
$3$-centralizer in  $\M(22)$. Since $Z$ is not weakly closed in $C_{C_S(\pi)}(\rho)$ with respect to $C_K(\rho)$,
we have that $Z$ is not weakly closed in $C_S(\rho)$ with respect to $C_G(\rho)$. Hence $C_G(\rho)/\langle
\rho\rangle$ is isomorphic to either $\M(22)$ or $\Aut(\M(22))$ by \cite[Theorem 2]{PSS}. Since $C_{H}(\rho)$ has
order $ 2^8\cdot 3^{10}$, we have that $C_G(\rho)/\langle \rho\rangle \cong \Aut(\M(22))$.  By Lemma~\ref{CR2} (ii),  $C_{S}(\rho)$
splits over $\langle \rho\rangle$ and therefore using a theorem  of Gasch\"utz \cite[(I.17.4)]{Huppert}
yields  that $C_G(\rho) \cong 3 \times \Aut(\M(22))$. Finally we note that $\rho$ is inverted in $H$ (by a
conjugate of $\sigma$ for example) and so the proof of the lemma is complete.
\end{proof}

\section{A $2$-local subgroup in $C_G(r)$}

In this section we locate  an extraspecial group $E$ of order $2^{23}$ and $+$-type  and prove that $N_G(E)/E
\cong \Co_2$. We have $C_K(Z) = C_Q(\pi)C_L(\pi)$ and $C_R(\pi )\cong 2^2\times \Q_8$ by
Lemma~\ref{c2}. Remember that $\wt K = K/\langle \pi \rangle \cong {}^2\E_6(2){:}2$.

By Proposition~\ref{2e62facts} (v) there is an involution $\wt r \in C_{\wt R}(\pi)\setminus Z(\wt R)$ such that $C_{\wt{C_{Q} (\pi)}}(\wt r) \cong 3^{1+4}_+$
and such that $\wt r$ is a $2$-central involution in $\wt K$. Since $C_R(\pi )\cong 2^2\times \Q_8 \le O_{3,2}(C_H(\pi))$ by Lemma~\ref{c2}, every involution in $O_{3,2}(C_H(\pi))$ is conjugate to an involution in $C_R(\pi)$ and so we can choose a preimage $r$  of $\wt r$ in $C_R(\pi)$  to be an involution.
 By  Proposition~\ref{2e62facts} (vi), $$C_{\wt K}(\wt r)\approx
2^{1+20}_+.\PSU_6(2){:}2$$
and  $C_{\wt K} (\wt r)$ acts absolutely irreducibly on $F^*(C_{\wt K}(\wt r))/\langle \wt r\rangle$.

Let $B \in \syl_2(N_{K}(\langle \pi, r\rangle))$ and  set $W= C_Q(\langle \pi, r\rangle)$. Then, by the choice of $r$, $W \cong 3^{1+4}_+$ and, by Lemma~\ref{CR}(i),
  $W= C_Q(r)$.
Taking $\rho \in W\setminus Z$, we have $\rho$ is not $H$-conjugate to $Z$ by Lemma~\ref{CR2} (i). It follows
that $C_{\wt{K}}(\wt \rho) \cong 3 \times \PSU_6(2).2$ and thus $C_G(\rho) \cong 3 \times \Aut(\M(22))$ by
Lemma~\ref{Crho}.

\begin{lemma}\label{ZT} We have $O_2(C_G(Z(B)))'=\langle r\rangle$ and $Z(B)= \langle r,\pi\rangle$.
Furthermore, if $J$ is the preimage of $ F^*(N_{\wt K}(Z(\wt
B)))$, then $J =\langle \pi \rangle J_1$ where $J_1$ is extraspecial of order $2^{21}$ and $+$-type.
\end{lemma}

\begin{proof}  We have that $J$ is the preimage of $$\wt J=O_2(N_{\wt K}(Z(\wt B)))= F^*(N_{\wt K}(Z(\wt
B))).$$ Therefore $\wt J$ is extraspecial of order $2^{21}$ and $+$-type. Assume that $J'$ has order $4$. Then, as
$N_{\wt K}(Z(\wt B))$ acts irreducibly on $J/J'$, we have $J'= Z(J)$ and then \cite[Lemma
2.73]{SymplecticAmalgams} implies that the representation of $\PSU_6(2)$ on $J/J'$ can be written over $\GF(4)$
and as such has dimension 10.  This however contradicts Proposition~\ref{2e62facts} (vi).

Thus $J'$ has order $2$. As $W\cong 3^{1+4}_+$ acts on $J/Z(B)$ and $W$ has no faithful representations of dimension less that $18$, we see that $C_{\wt J}(\wt Z) \cong \Q_8$. Since $\wt W$ is normalized by $C_{\wt J}(\wt Z)$, we obtain $C_{\wt J}(\wt W)=C_{\wt J}(\wt Z)\cong \Q_8$. Now
$C_J(W) \le C_K(Z)= C_H(\pi)$ and $C_J(W)'\le C_{QR}(\pi)$. It follows that $r \in C_J(W)'
$ and this proves the first part of the result.

We have $\Phi(J)\le \langle \pi, r\rangle$. If $\Phi(J)\neq J'$, then $J/J'$ must have exponent $4$ with some
element squaring to $\pi \langle r\rangle$. But then $\Omega_1(J/\langle r\rangle)$ contains $\pi\langle
r\rangle$ and has index $2$ in $J/\langle r\rangle$ and this contradicts $N_{\wt K}(Z(\wt B))$ acting irreducibly
on $J/\langle \pi ,r\rangle$. Thus $\Phi(J)=\langle r\rangle= J'$ and $Z(J) =\langle \pi, r\rangle$. It follows
that $J =\langle \pi \rangle J_1$ where $J_1$ is extraspecial of order $2^{21}$.
\end{proof}

\begin{lemma}\label{not2central} We have that $\pi$ is not $2$-central and if $|B_1 : B| = 2$, then $[B_1,r] = 1$. In particular $r$ is $2$-central. \end{lemma}

\begin{proof} Suppose that $B \in \syl_2(G)$. Then any fusion between elements of $Z(B)$ occurs in $N_G(B)$. Hence
either no two elements of $Z(B)^\#$ are conjugate or they are all conjugate. Since $\langle r\rangle=
O_2(C_G(Z(B)))'$ by Lemma~\ref{ZT}, $r$ is not conjugate to either $\pi$ or $\pi r$. Hence $\pi$ and $\pi r$ are
not conjugate in $G$.

By construction  $C_{K}(\rho)$ involves $\PSU_6(2)$. Hence, by \cite[Table $\M(22)$, page 251]{Aschbacher3T} and
Lemma~\ref{Crho}, $\pi$ is a 3-transposition in $E(C_G(\rho))\cong \M(22)$. Now from the choice of $Z(B)$, the
centralizer of $\rho$ in $C_G(Z(B))$ involves $\SU_4(2)$. Hence we have that $C_{E(C_G(\rho))}(r)$ has
shape $2^{(1+1)+8}.\SU_4(2).2$ with centre $\langle \pi, r\rangle$. Since $E(C_G(\rho))\cong \M(22)$ has
exactly three conjugacy classes of involutions and since one for of these classes the centralizer does not involve
$\SU_4(2)$ (see \cite[Table 5.3t]{GLS3}), we deduce that $\pi$ and $\pi r$ are conjugate in $F^*(C_G(\rho))$.
With this contradiction we have $B   \not \in \syl_2(G)$. Now let $|B_1: B| = 2$, then $B_1$ acts on $Z(B)$ and
so on $O_2(C_G(Z(B)))^\prime$. So $B_1$ centralizes $r$ by Lemma~\ref{ZT} and so we conclude $r$ is $2$-central.
\end{proof}

\begin{lemma} \label{E1}There is an extraspecial group $E$ of order $2^{23}$ and plus type containing $\pi$
such that $N_{K}(E)E/E \cong \PSU_6(2){:}2$.
\end{lemma}

\begin{proof} By Lemma~\ref{ZT},
$C_{K}(r)/\langle r \rangle$ is an extension of an elementary abelian group of order $2^{21}$ by
$\PSU_6(2){:}2$. Moreover, we have $C_K(r)= C_G(\langle \pi, r\rangle)=  C_G(Z(B))$. Now, by Lemma
\ref{not2central}, there is a group $B_1$ with $|B_1 : B| = 2$ and $[B_1,r] = 1$. In particular, $B_1$
normalizes $Z(B)$ and therefore $B_1$ normalizes $C_K(r)=C_G(Z(B))$ and $J=O_2(C_K(r))$. Thus $B_1C_K(r)/J$ has a
normal subgroup of index $4$ isomorphic to $\PSU_6(2)$. As $\Out(\PSU_6(2))$ is not divisible by $4$ and as
$C_K(r)/J \cong \PSU_6(2){:}2$, we infer that $|O_2(B_1C_K(r))/J|=2$. Set $E= O_2(B_1C_K(r))$. Then,  as $J/\langle
\pi, r \rangle$ is an irreducible module for $C_K(r)/J\cong \PSU_6(2){:}2$ and $[Z(J),E] =\langle r\rangle$, we
have $[J/\langle r\rangle, E]=1$. Let $J_1$ be as in Lemma~\ref{ZT}. Then $[J_1,E]\le \langle r\rangle$. Hence
\cite[23.8]{AschbacherFG} implies that $E= C_{E}(J_1)J_1$. Now we have that $C_{E}(J_1)$ contains $\langle \pi,
r\rangle$ and is non-abelian. It follows that $C_{E}(J_1)$ is a dihedral group of order $8$ and that $E$ is an
extraspecial group of order $2^{23}$ and +-type. Finally we note that $N_K(E)E = N_K(E)B_1$ and have the result.
\end{proof}

\begin{lemma}\label{signalizer} The following hold.
\begin{enumerate}

\item $C_E(Z) = C_E(W) =O_2(C_H(r))$ is extraspecial of order $2^5$ and is the unique maximal signalizer for $W$ in
$C_H(r)$; and
\item $C_E(\rho)$ has order $2^{11}$ and  $C_{N_G(E)/E}(\rho E  ) \cong 3 \times \SU_4(2){:}2$.
\end{enumerate}
In particular,  $E$  is the unique maximal $W$-signalizer in $C_G(r)$.
\end{lemma}

\begin{proof}  Since $W$ is extraspecial of order $3^{1+4}_+$, and the smallest faithful $\GF(2)$-representation of $W$
has dimension 18 we have that $C_E(Z) $ is extraspecial of order $2^5$. Now $[E,W]\le O_2(C_K(r))$ and
$\langle r,\pi\rangle$ is centralized by $W$. Thus, as $C_{J/\langle \pi \rangle}(W) \cong \Q_8$, we have
$C_E(W)=C_E(Z)$. From the choice of $r$, we have $r \in R$. Hence we may apply Lemma~\ref{CR} (iii). From there we see
that $C_E(W) = O_2(C_H(r)) \cong 2^{1+4}_-$ is the unique maximal signalizer for $W$ in $C_H(r)$. Hence (i)
holds.

%
%

 Since $r$ commutes with $\rho$,  $E/C_E(W)$ admits $\langle \rho, Z \rangle$ with $Z$ fixed-point-free and
 therefore, as $|E/C_E(W)|= 2^{18}$ and all the subgroups  of $\langle \rho,Z\rangle$ other than $Z$ are conjugate in $W$,
 $C_E(\rho)$ has order $2^{11}$. Moreover,  by the
choice of $\rho$, we have that $C_E(\rho)$ is normalized by $\SU_4(2)$ which is involved in $C_{C_G(Z(B))}(\rho)$. The
structure of $C_G(\rho)$ as given in Lemma~\ref{Crho} and \cite[Table $\M(22)$, page 250]{Aschbacher3T}  show
that $E(C_{N_G(E)}(\rho ) /C_E(\rho)) \cong \SU_4(2)$, which is (ii).  Since in $\SU_4(2)$ an extraspecial group
of order 27 does not normalize any non-trivial $3^\prime$-groups, we have that  $3 \times \SU_4(2)$ has no
non-trivial $C_W(\rho)$-signalizers. Hence $C_E(\rho)$ is the unique maximal $C_W(\rho)$-signalizer in $C_G(r)$.
Suppose that $I$ is an $W$-signalizer. Then $I$ is generated by $W$-signalizers in $C_{C_G(r)}(Z) $ and
$W$-conjugates of $C_W(\rho)$-signalizers. Thus $I \le E$ by (i) and (ii).
\end{proof}

%
%

\begin{proposition}\label{NE} We have that $N_G(E)/E \cong \Co_2$. In particular, $C_{C_G(r)}(Z) \le N_G(E)$.
\end{proposition}

\begin{proof} By Lemma \ref{signalizer} we have that $N_{C_G(r)}(W)$ normalizes $E$ and from Lemma~\ref{CR} (ii) we have
$N_{C_G(r)}(W)/O_2(C_H(r)) $ is similar to a $3$-centralizer in $\Co_2$. Since $O_2(C_H(r)) \le E$ by
Lemma~\ref{signalizer}(i), we have $C_{N_G(E)/E}(ZE/E)$ is similar to a $3$-centralizer in $\Co_2$.
Because $ZE/E$ is not weakly closed in a Sylow $3$-subgroup of
  $C_{N_G(E)/E}(\rho E)$, $ZE/E$ is not weakly closed in  a Sylow $3$-subgroup of  $N_G(E)/E$. Therefore using
  \cite[Theorem 1.1]{ParkerRowley} we get that $N_G(E)/E \cong
\Co_2$. Since $C_{C_G(r)}(Z) = N_{C_G(r)}(W)$, we also have $C_{C_G(r)}(Z) \le N_G(E)$.
\end{proof}

\begin{lemma}\label{strong}  $N_G(E)$ is strongly $3$-embedded in $C_G(r)$. In particular, $N_G(E)$ controls fusion of $3$-elements in $N_G(E)$.
\end{lemma}

\begin{proof} By \cite[Table 5.3k]{GLS3} $\Co_2$ has exactly two conjugacy classes of elements of order
three. Hence the same holds for $N_G(E)$ by Proposition~\ref{NE}. Moreover as the non-trivial elements of
$Z$ and $\rho$  are not $G$-conjugate (by Lemma~\ref{Crho} for example) and, as $\langle Z,\rho\rangle \le
N_G(E)$, we have that $z \in Z^\#$ and $\rho$ can be taken as representatives  of the $N_G(E)$  conjugacy classes
of elements of order $3$. By Proposition~\ref{NE} we know $C_{C_G(r)}(Z) \le N_G(E)$.  By Lemma \ref{Crho},
$C_G(\rho) \cong 3 \times \Aut(\M(22))$. Hence $C_{C_G(r)}(\rho)$ has shape $ 3 \times 2^{1+10}_+.\SU_4(2).2$ and
therefore Lemma~\ref{signalizer} (ii) implies that $C_{C_G(r)}(\rho) \le N_G(E)$. Hence $N_G(E)$ is strongly
$3$-embedded in $C_G(r)$. The last claim follows from \cite[Proposition 17.11]{GLS2}.
\end{proof}

\section{The centralizer of $r$}

In this section we will show that $C_G(r) = N_G(E)$. For this we set $\ov {C_G(r)} = C_G(r)/\langle r \rangle$. Recall that by Proposition \ref{NE} we have that $N_G(E)/E \cong \Co_2$.

\begin{lemma}\label{JS} $\ov{N_G(E)}$ controls fusion in $\ov E$.
\end{lemma}

\begin{proof} This comes from Lemma~\ref{Co2}(iv) and \cite[37.6]{AschbacherFG}.
\end{proof}

%
%
%
%
%
%
%
\begin{lemma}\label{cent3} Let $\ov v \in \ov {N_G(E)} \setminus \ov E$ be an involution. Then $3$ divides $|C_{\ov N_G(E)}(\ov v)|$.
\end{lemma}

\begin{proof} This follows from Lemma~\ref{Co2}(iii).
\end{proof}

\begin{theorem}\label{centralizer2}$C_G(r) = N_G(E)$. In particular,  $C_G(r)$ is an extension of an
extraspecial group of order $2^{23}$ of $+$-type by $\Co_2$.
\end{theorem}

\begin{proof} Let $g \in \ov{C_G(r)}$ such that $\ov \pi^g \in \ov{N_G(E)} \setminus\ov E$. Set $\ov v = \ov \pi^g$.
By Lemma \ref{cent3}, there is  an element $\tau$ of order $3$ in $C_{\ov{N_G(E)}/\ov E}(\ov v)$. Hence $\tau^{g^{-1}} \in
C_{\ov{C_G(r)}}(\pi) \leq \ov{N_G(E)}$ (recall we know $C_G(\pi)$ by Theorem~\ref{tpi2}). By Lemma \ref{strong} there is some $h \in \ov{N_G(E)}$ such that $\tau^{g^{-1}h} =
\tau$. Further by Lemma \ref{strong} we have that $C_{\ov{C_G(r)}}(\tau) \leq \ov{N_G(E)}$ and so $\ov{g^{-1}h} \in  \ov{N_G(E)}$.
But then also $g \in \ov{N_G(E)}$, which contradicts $\ov \pi \in \ov E$ but $\ov v \not\in \ov E$. So, because of Lemma \ref{JS},  we
have  demonstrated the following two properties of the embedding of $N_G(E)$ in $C_G(r)$:
\begin{enumerate}\item  [(a)] $C_{\ov{C_G(r)}}(\ov \pi) \le \ov{N_G(E)}$; and
\item [(b)]$\ov \pi^{\ov{C_G(r)}} \cap \ov{N_G(E)} = \ov \pi^{\ov{N_G(E)}}$.
\end{enumerate}
Suppose that $\ov u \in \ov E$ is an involution and that $\ov u$  is $\ov{C_G(r)}$-conjugate to $\ov w \in
\ov{N_G(E)}\setminus \ov E$. Then, as every involution of $\Co_2$
 is conjugate to an involution in $C_K(r)E/E\cong \PSU_6(2){:}2$ by Lemmas~\ref{Co2} (ii) and \ref{E1},
 we may suppose that $\ov w \in C_{\ov{C_G(r)}}(\ov{\pi})$. Points (a) and (b) above provide the hypothesis of Lemma~\ref{fusion}. Therefore
 $$\ov w \in u^{\ov{C_G(r)}}\cap C_{\ov{C_G(r)}}(\ov \pi)= \ov u^{\ov{N_G(E)}} \cap C_{\ov{C_G(r)}}(\ov \pi).$$ Since $\ov u \in \ov E$ and $\ov w \not \in \ov E$ this is impossible.

We have shown that $\ov E$ is strongly 2-closed in $\ov {N_G(E)}$. Since $O_{2^\prime}(C_G(r)) = 1$ by Lemma \ref{signalizer}, we now apply Lemma~\ref{Gold} to obtain $C_G(r)=N_G(E)$.\end{proof}

We can now prove our second main theorem.

\begin{proof}[Proof of Theorem~\ref{Main2}] By Theorems \ref{tpi2} and \ref{centralizer2},  $C_G(\pi)$
is isomorphic to $(2\udot {}^2\E_6(2)){:}2$ and $C_G(r)$ is an extension of an extraspecial group of order
$2^{23}$ of $+$-type by $\Co_2$. Hence we have verified our definition of $\F_2$ as required to prove the
theorem.
\end{proof}

\end{document}